\documentclass[11pt]{article}

 \usepackage{amsmath,amssymb,amscd,amsthm,esint}
 
\usepackage{graphics,amsmath,amssymb,amsthm,mathrsfs}

\usepackage{graphics,amsmath,amssymb,amsthm,mathrsfs,amsfonts}

\usepackage{sidecap}
\usepackage{float}
\usepackage{extarrows}
\usepackage{booktabs}
\usepackage{verbatim}
\usepackage{hyperref}
\usepackage[usenames,dvipsnames]{xcolor}

\newtheorem{thm}{Theorem}[section]
\newtheorem{lemma}[thm]{Lemma}
\newtheorem{cor}[thm]{Corollary}

\theoremstyle{definition}
\newtheorem{remark}[thm]{Remark}

\def\XXint#1#2#3{{\setbox0=\hbox{$#1{#2#3}{\int}$}
         \vcenter{\hbox{$#2#3$}}\kern-.5\wd0}}

\def\R{\mathbb{R}}

\def\e{\varepsilon}

\def\tY{\widetilde{Y}}

\numberwithin{equation}{section}

\begin{document}

\title{Uniform $W^{1, p}$ Estimates and Large-Scale Regularity \\
 for Dirichlet Problems  in Perforated Domains}

\author{
Zhongwei Shen\thanks{Supported in part by NSF grants DMS-1856235, DMS-2153585, and by Simons Fellowship.}   
 \qquad Jamison Wallace\thanks{Supported in part by NSF grants DMS-1856235 and  DMS-2153585.} 
}
\date{}
\maketitle

\begin{abstract}

In this paper we study the Dirichlet problem for Laplace's equation in a domain
$\omega_{\e, \eta}$  perforated  periodically with small holes
in $\mathbb{R}^d$, where $\e$ represents the scale of the minimal distances between holes and $\eta$ the ratio between the 
scale of sizes of holes and $\e$.
We establish $W^{1, p}$ estimates for solutions with bounding constants depending explicitly on 
$\e$ and $\eta$. The proof relies on a large-scale Lipschitz estimate for harmonic functions
in perforated domains. The results are optimal for $d\ge 2$.

\medskip


\medskip

\noindent {\it MR (2020) Subject Classification}: 35Q35;  35B27; 76D07.

\end{abstract}


\section{Introduction}\label{section-1}

This paper continues the study of the Dirichlet problem for Laplace's equation,
\begin{equation}\label{D-01}
\left\{
\aligned
-\Delta u & =F+ \text{\rm div} (f) & \quad &  \text{ in } \omega_{\e, \eta} ,\\
u & = 0 & \quad & \text{ on } \partial \omega_{\e, \eta},
\endaligned
\right.
\end{equation}
in a domain $\omega_{\e, \eta}$ perforated with small holes,
where $\e$ represents the scale of the minimal distances  between holes and $\eta $ the ratio  between  the sizes of holes and $\e$.
We are interested in the $W^{1, p} $ estimates,
\begin{equation}\label{A-B}
\|\nabla u\|_{L^p (\omega_{\e, \eta})}
\le A_p (\e, \eta) \| f\|_{L^p(\omega_{\e, \eta})}
+ B_p (\e, \eta) \| F \|_{L^p(\omega_{\e, \eta})},
\end{equation}
and
\begin{equation}\label{C-D}
\| u\|_{L^p (\omega_{\e, \eta})}
\le C_p (\e, \eta) \| f\|_{L^p(\omega_{\e, \eta})}
+ D_p (\e, \eta) \| F \|_{L^p(\omega_{\e, \eta})},
\end{equation}
for $1< p< \infty$,
with bounding constants $A_p(\e, \eta)$, $B_p(\e, \eta)$, $C_p (\e, \eta)$ and $D_p(\e, \eta)$
depending explicitly on the small parameters $\e, \eta \in (0, 1]$.

To state our main results, we let $Y=[-1/2,1/2]^d $ be a closed unit cube in $\R^d$ and $T$ 
 the closure of an open subset of $Y$.
Throughout the paper we shall assume that 
$Y\setminus  {T}$ is connected and that
\begin{equation}\label{condition-0}
B(0, c_0) \subset T\quad \text{ and } 
  \quad \text{\rm dist} (\partial T, \partial Y)\ge c_0>0
\end{equation}
for some $c_0>0$.
Define
\begin{equation}\label{omega}
\omega_{\e, \eta} = \R^d \setminus \bigcup_{k\in \mathbb{Z}^d}  \e ( k + \eta {T}),
\end{equation}
where $0<\e,  \eta\le 1$.
Roughly speaking, the periodically perforated domain $\omega_{\e, \eta}$ is obtained from 
$\R^d$ by removing  a hole $\e (k+\eta {T})$ of size  $\e \eta$ from each cube $\e (k+Y)$ of size $\e$.
The distances between holes  are bounded below by $c_0\e$.

The following are the main results in this paper.
The first  theorem deals with the case $d\ge 3$, while the second treats the case $d=2$.

\begin{thm}\label{main-thm-1}
Suppose $d\ge 3$ and  $1< p< \infty$.
Let $\omega_{\e, \eta}$ be given by \eqref{omega}, where $T$ is the closure of an open subset of $Y$ with $C^1$ boundary.
For any  $f\in L^p(\omega_{\e, \eta}; \R^d)$ and $F \in L^p(\omega_{\e, \eta})$, the Dirichlet problem \eqref{D-01}
has a unique solution $u$  in $W_0^{1, p} (\omega_{\e, \eta})$.
Moreover, the solution satisfies the estimate,
\begin{equation}\label{m-e-1}
\| \nabla u \|_{L^p(\omega_{\e, \eta} )} \le
\left\{
\aligned
& 
 C \eta^{-d |\frac12 -\frac{1}{p}|} \| f\|_{L^p(\omega_{\e, \eta})}
+C \e \eta^{1-\frac{d}{2}} \| F \|_{L^p(\omega_{\e, \eta})} & \quad & \text{ for } 1< p< 2,\\
&  C \eta^{-d |\frac12 -\frac{1}{p}|} \| f\|_{L^p(\omega_{\e, \eta})} 
 +C \e \eta^{1-d +\frac{d}{p}} \| F \|_{L^p(\omega_{\e, \eta})}
& \quad & \text{ for } 2\le  p< \infty,\\
\endaligned
\right.
\end{equation}
and
\begin{equation}\label{m-e-2}
\|  u \|_{L^p(\omega_{\e, \eta} )} \le
\left\{
\aligned
&C \e \eta^{1-\frac{d}{p}}  \| f\|_{L^p(\omega_{\e, \eta})} 
+ C \e^2 \eta^{2-d}  \| F \|_{L^p(\omega_{\e, \eta})}
& \quad & \text{ for } 1< p< 2,\\
& C \e \eta^{1-\frac{d}{2}}  \| f\|_{L^p(\omega_{\e, \eta})} 
+ C \e^2 \eta^{2-d} \| F \|_{L^p(\omega_{\e, \eta})} & \quad & \text{ for } 2\le  p< \infty,\\
\endaligned
\right.
\end{equation}
where $C$ depends on $d$, $p$ and $T$.
Furthermore, the estimates \eqref{m-e-1}-\eqref{m-e-2}  are sharp.
\end{thm}

\begin{thm}\label{main-thm-2}
Suppose $d= 2$ and  $1< p< \infty$.
Let $\omega_{\e, \eta}$ be given by \eqref{omega}, where $T$ is the closure of an open subset of $Y$ with $C^1$ boundary.
For any  $f\in L^p(\omega_{\e, \eta}; \R^2)$ and $F \in L^p(\omega_{\e, \eta})$, the Dirichlet problem \eqref{D-01}
has a unique solution $u$  in $W_0^{1, p} (\omega_{\e, \eta})$.
Moreover, the solution satisfies the estimate,
\begin{equation}\label{m-e-1a}
\| \nabla u \|_{L^p(\omega_{\e, \eta} )} \le
\left\{
\aligned
& 
 C \eta^{-2 |\frac12 -\frac{1}{p}|} |\ln (\eta/2)|^{-\frac12}  \| f\|_{L^p(\omega_{\e, \eta})}
+C \e |\ln (\eta/2)|^{\frac12}  \| F \|_{L^p(\omega_{\e, \eta})} & \quad & \text{ for } 1< p< 2,\\
&  \| f\|_{L^2(\omega_{1, \eta})} +  C \e |\ln (\eta/2)|^{\frac12}   \| F \|_{L^2(\omega_{\e, \eta})} & \quad & \text{ for } p=2,\\
&  C \eta^{-2 | \frac12-\frac{1}{p}|}  |\ln (\eta/2)|^{-\frac12} \| f\|_{L^p(\omega_{\e, \eta})} 
 +C \e \eta^{-1 +\frac{2}{p}} \| F \|_{L^p(\omega_{\e, \eta})}
& \quad & \text{ for } 2< p< \infty,\\
\endaligned
\right.
\end{equation}
and
\begin{equation}\label{m-e-2a}
\|  u \|_{L^p(\omega_{\e, \eta} )} \le
\left\{
\aligned
&C \e \eta^{1-\frac{2}{p}}  \| f\|_{L^p(\omega_{\e, \eta})} 
+ C \e^2  |\ln (\eta/2)| \| F \|_{L^p(\omega_{\e, \eta})}
& \quad & \text{ for } 1< p< 2,\\
& C \e |\ln (\eta/2)|^{\frac12}   \| f\|_{L^p(\omega_{\e, \eta})} 
+ C \e^2  |\ln (\eta/2)| \| F \|_{L^p(\omega_{\e, \eta})} & \quad & \text{ for } 2\le  p< \infty,\\
\endaligned
\right.
\end{equation}
where $C$ depends on  $p$ and $T$.
Furthermore, the estimates \eqref{m-e-1a}-\eqref{m-e-2a} are sharp.
\end{thm}

We point out that the estimates  \eqref{m-e-1}-\eqref{m-e-2a} are sharp in $\e$ and $\eta$.
Indeed,  if $d\ge 3$ and the estimates \eqref{A-B} and \eqref{C-D} hold for some 
constants $A_p(\e, \eta)$, $B_p(\e, \eta)$, $C_p(\e, \eta)$ and $D_p(\e, \eta)$,
then 
\begin{equation}\label{A-low}
\aligned
A_p (\e, \eta)  & \ge c\,  \eta^{-d|\frac12-\frac{1}{p}|},\\
D_p (\e, \eta) & \ge c \, \e^2 \eta^{2-d},
\endaligned
\end{equation}
for $1< p< \infty$, 
\begin{equation}\label{B-low}
B_p (\e, \eta)\ge 
\left\{
\aligned
& c\,  \e \eta^{1-\frac{d}{2}} & \quad & \text { for } 1< p\le  2,\\
& c\,  \e \eta^{1-d +\frac{d}{p}}  & \quad &  \text{ for  } 2< p< \infty,
\endaligned
\right.
\end{equation}
and
\begin{equation}\label{C-low}
C_p (\e, \eta)\ge 
\left\{
\aligned
& c\,  \e \eta^{1-\frac{d}{p}} & \quad & \text { for } 1< p< 2,\\
& c\,  \e \eta^{1-\frac{d}{2}}  & \quad &  \text{ for  } 2\le p< \infty,
\endaligned
\right.
\end{equation}
where  $c>0$ depends only on $d$, $p$ and $T$.
Similar statements hold for the case $d=2$; the lower bounds for $A_p(\e, \eta)$, $B_p(\e, \eta)$, $C_p(\e, \eta)$ and
$D_p(\e, \eta)$ are given by the corresponding constants in \eqref{m-e-1a}-\eqref{m-e-2a} (with a different $c>0$).
The powers of $\e$ in the estimates \eqref{m-e-1}-\eqref{m-e-2a}
 are due to scaling. In fact, by rescaling,  it suffices to prove Theorems \ref{main-thm-1} and \ref{main-thm-2} for the case $\e=1$.
The sharpness in $\eta$ was proved by the first author in \cite{Shen-2022}, using a $Y$-periodic function $\chi_\eta$ that
satisfies 
$$
-\Delta \chi_\eta =\eta^{d-2} \quad \text{ in  } \omega_{1, \eta} \quad \text{ and } \quad
\chi_\eta =0 \quad  \text{ in } \R^d \setminus \omega_{1,\eta}.
$$
In \cite{Shen-2022} we  also established estimates \eqref{A-B} and \eqref{C-D} in a general non-periodic setting 
with a sharp constant  $D_p(\e, \eta)= C \e^2 \eta^{2-d}$ and almost sharp constants for $A_p(\e, \eta)$, $B_p(\e, \eta)$ and $C_p(\e, \eta)$
(up to an arbitrary small  power of $\eta$).
The main results in this paper provide a complete solution in the periodic setting for $d\ge 2$.

We now describe our approach to Theorems \ref{main-thm-1} and \ref{main-thm-2}.
By rescaling we assume $\e=1$.
Our starting point is  the estimates \eqref{m-e-2} and \eqref{m-e-2a}  for $u$
in  the case $2\le p<\infty$.
The estimates were proved in a general non-periodic setting in \cite{Shen-2022}, using a classical method of test functions and a Poincar\'e 
inequality for $W^{1, 2} (Q_1)$  functions  that vanish on $Q_1\setminus \eta T$, where $Q_R=(-R/2, R/2)^d$.
To establish the estimates \eqref{m-e-1} and \eqref{m-e-1a}   for $\nabla u$,
we consider the Dirichlet problem in a weighted Sobolev space for Laplace's equation $-\Delta u = F +\text{\rm div}(f)$ in 
an exterior domain $\R^d\setminus T$.
By localization and rescaling,  this allows us to control the $L^p$ norm of $\nabla u$ in each cell 
$k+Q_1$ by the $L^p$ norm of $u$ in a slightly larger cell $k + (1+c_0)Q_1$.

Next, to bound the localization error,
we  construct a corrector $\psi_\eta \in W^{1, 2}(Q_1)$ such that
\begin{equation}\label{psi-0}
\left\{
\aligned
-\Delta \psi_\eta & = F_\eta +\text{\rm div}(f_\eta)  & \quad & \text{ in } Q_1 \setminus \eta T,\\
\psi_\eta & =0 & \quad & \text{ in } \eta T,\\
\psi_\eta & =1 & \quad & \text{ in } Q_1\setminus B(0, 1/3),
\endaligned
\right.
\end{equation}
where $F_\eta$ and $f_\eta$ satisfy the condition $\| f_\eta \|_\infty +\| F_\eta\|_\infty \le C \eta^{d-2}$ for $d\ge 3$.
The construction of $\psi_\eta$, which is motivated by the correctors used in  \cite{Allaire-90a}, uses a solution to the exterior problem,
\begin{equation}\label{ext-0}
\left\{
\aligned
\Delta \phi_* & =0 & \quad & \text{ in } \R^d \setminus T, \\
\phi_* & =0 & \quad & \text{ in } T,\\
\phi_* & \to 1 &\quad &  \text{ as } |x| \to \infty.
\endaligned
\right.
\end{equation}
See \eqref{psi} for $d\ge 3$ and \eqref{2d-p} for $d=2$.
We  apply a localization argument to the solution  $u-\alpha \psi_\eta$ in $(1+c_0) Q_1 \setminus \eta T$, with
$$
\alpha =\fint_{(1+c_0)Q_1 \setminus B(0, 1/3)} u.
$$
With sharp estimates for $\|u \|_{L^p(\omega_{1, \eta})}$ and $\psi_\eta$,
this reduces the $L^p$ estimate of $\nabla u$ to the $L^p$ estimate for  the operator  $S_{\e, \eta}$, defined by
\begin{equation}\label{S-0a}
S_{\e, \eta} (F, f) (x) 
=\left(\fint_{x+ \e Q_2} |\nabla u|^2 \right)^{1/2},
\end{equation}
for $p>2$ ($u$ is extended by zero into the holes).
Note that $\| S_{\e, \eta} (F, f)  \|_{L^2(\R^d)} = \|\nabla u\|_{L^2(\R^d)}$.
By a real-variable argument in \cite{Shen-2005}, 
to establish the $L^p$ boundedness of the operator for $p>2$, it suffices to prove a  (weak)  reverse H\"older inequality 
  in a cube $Q$ for solutions 
of  the Dirichlet problem \eqref{D-01} with $F=0$ and $f=0$ in $4Q$.

Finally, we note that the desired reverse H\"older inequality in $L^p$ follows from a large-scale Lipschitz estimate,
\begin{equation}\label{Lip-0}
\sup_{1\le r \le R}
\left(\fint_{Q_r} |\nabla u|^2 \right)^{1/2}
\le C \left(\fint_{Q_R} |\nabla u|^2 \right)^{1/2},
\end{equation}
 for harmonic functions in 
perforated domains. By exploiting the periodicity of the domain $\omega_{1, \eta}$, 
we are able to establish the large-scale Lipschitz estimate, using an approach  taken from \cite{AKS-2020}.
The proof relies on a Caccioppoli inequality as well as a discrete Sobolev inequality in $\mathbb{Z}^d$.
The constant $C$ in \eqref{Lip-0} depends only on $d$ and $c_0$ in \eqref{condition-0}.

The paper is organized as follows.
In Section \ref{section-L} we establish a large-scale $L^\infty$ estimate for harmonic functions in $Q_R\cap \omega_{1, \eta}$
that vanish on $Q_R\cap \partial \omega_{1, \eta}$.
The large-scale Lipschitz estimate \eqref{Lip-0}  is proved in Section \ref{section-Lip}.
The $L^p$ bound with $2<p<\infty$ for  the operator $S_{\e, \eta}$ in \eqref{S-0a}   is obtained  in Section \ref{section-W}.
In Sections \ref{section-loc1} and \ref{section-loc2} we present the localization argument for solutions in $(1+c_0)Q_1 \setminus \eta T$.
The argument relies on  some weighted estimates in \cite{AGG-1997}  for an exterior problem and utilizes  the corrector $\psi_\eta$ mentioned before.
Finally, the proofs of Theorems \ref{main-thm-1} and \ref{main-thm-2} are given in Section \ref{section-p}.



\section{Large-scale $L^\infty$ estimates}\label{section-L}
 
Throughout this section we assume that $\omega_{\e, \eta}$ is given by \eqref{omega}, where
 $T$ is the closure of  an open subset of $Y$ with Lipschitz boundary.
Let $Q_R=(-R/2, R/2)^d$.
Our goal is to prove the following theorem.

\begin{thm}\label{L-thm}
Let $u\in W^{1, 2}(Q_R)$ for some $R\ge \e$.
Suppose that
\begin{equation}\label{5.0-L}
\Delta u =0 \quad \text{ in } \quad Q_R \cap \omega_{\e, \eta} \quad \text{ and } \quad
u=0 \quad \text{ in } Q_R \setminus  \omega_{\e, \eta}.
\end{equation}
Then, for $\e\le r\le R$,
\begin{equation}\label{5.0-1}
\left(\fint_{Q_r} |u|^2 \right)^{1/2}
\le C \left(\fint_{Q_R} |u|^2 \right)^{1/2},
\end{equation}
where  $C$ depends on $d$.
\end{thm}

The proof of Theorem \ref{L-thm}, as well as the proof of Theorem \ref{Lip-thm} in the next section,
is based on an approach found  in \cite{AKS-2020} and
relies on a Caccioppoli inequality for solutions of \eqref{5.0-L}.

For $u\in L^1(\R^d)$ and $z\in \mathbb{Z}^d$, define
\begin{equation}\label{hat}
\hat{u} (z) =\int_{z+Q_1} u(x)\, dx.
\end{equation}

\begin{lemma}\label{dp-lemma}
Let $u\in W^{1, 2} (Q_{r+2})$, where $r\ge 1$.
Then
\begin{equation}\label{dp}
\left(\fint_{Q_{r}} |u|^2 \right)^{1/2}
\le C \max_{z\in \mathbb{Z}^d \cap  {Q_{r+2}}} | \hat{u}(z)|
+ C \left(\fint_{Q_{r+2}} |\nabla u|^2 \right)^{1/2},
\end{equation}
where $C$ depends only on $d$.
\end{lemma}

\begin{proof}
Let $z\in \mathbb{Z}^d\cap Q_r$. By Poincar\'e's inequality,
\begin{equation}\label{P00}
\int_{z+ Q_1} | u|^2 \, dx  \le C |\hat{u}(z)|^2 +C \int_{z+Q_1} |\nabla u|^2\, dx,
\end{equation}
where $C$ depends only on $d$.
Let $\ell\ge 1$ be an odd integer.
By summing \eqref{P00} over $z\in \mathbb{Z}^d \cap Q_\ell $, we obtain
\begin{equation}\label{P0a}
\left(\fint_{Q_{\ell}} |u|^2 \right)^{1/2}
\le C \max_{z\in \mathbb{Z}^d \cap  {Q_{\ell}}} | \hat{u}(z)|
+ C \left(\fint_{Q_{\ell}} |\nabla u|^2 \right)^{1/2}.
\end{equation}
Finally, for any $r\ge 1$, choose an odd integer $\ell$ such that $r\le \ell \le r+2$.
It is not hard to see that \eqref{dp} follows from \eqref{P0a}.
\end{proof}

For a function  $g$ defined in $\R^d$ or $\mathbb{Z}^d$,
let
\begin{equation}\label{Delta}
\Delta_j g (x)= g(x+e_j) - g(x)
\end{equation}
for $1\le j \le d$, where $e_j=(0, \dots, 1, \dots, 0)$ with $1$ in the $j$th position.
For a multi-index $\gamma=(\gamma_1, \gamma_2, \dots, \gamma_d)$, 
we use the notation $\Delta^\gamma g=g$ if $\gamma=0$, and
$$
\Delta^\gamma g = \Delta_1^{\gamma_1} \Delta_2^{\gamma_2} \cdots \Delta_d^{\gamma_d} g
$$
if $|\gamma|\ge 1$.
Let $\partial^\ell  g = (\Delta^\gamma g  )_{|\gamma|=\ell}$ and
\begin{equation}
|\partial^\ell g |
=\left(\sum_{|\gamma|=\ell} |\Delta^\gamma g |^2 \right)^{1/2}
\end{equation}
for an integer $\ell\ge 0$.
It is not hard to see that 
\begin{equation}\label{FTC-1}
|\partial^{\ell+1}  \hat{u}(z) |
\le \left(\int_{z+3Q_1} |\nabla \partial^\ell u|^2 \, dx \right)^{1/2}
\end{equation}
for any $z\in \mathbb{Z}^d$.

The next lemma provides a discrete Sobolev inequality in $\mathbb{Z}^d$.

\begin{lemma}\label{ds-lemma}
Let $g$ be a function on $\mathbb{Z}^d$.
Then, for $R\ge 3d$, 
\begin{equation}\label{ds-1}
\max_{z\in \mathbb{Z}^d \cap {Q}_{R}} |g(z)|
\le C \sum_{\ell=0}^N R^\ell
\left(\frac{1}{R^d}
\sum_{z\in \mathbb{Z}^d \cap {Q}_{2R}}
|\partial^\ell g(z)|^2 \right)^{1/2},
\end{equation}
where $N=[d/2]+1$ and $C$ depends only on $d$,
\end{lemma}

\begin{proof}
This follows from  \cite[Lemma 2.6]{Stevenson-1991}.
\end{proof}

The following lemma gives the Caccioppoli inequality for solutions of 
$-\Delta u=F$ in $Q_r\cap \omega_{1, \eta}$ with $u=0$ on $Q_r\setminus \omega_{1, \eta}$.

\begin{lemma}\label{C-lemma}
Let $u \in W^{1, 2}(Q_r)$ for some $r\ge 1$.
Suppose that $-\Delta u=F$ in $Q_r \cap \omega_{1, \eta}$ and $u=0$ in  $Q_r \setminus \omega_{1, \eta}$.
Then
\begin{equation}\label{C}
\left(\fint_{Q_{sr}} |\nabla u|^2 \right)^{1/2}
\le \frac{C}{(t-s) r}
\left(\fint_{Q_{tr}} |u|^2 \right)^{1/2}
+ C (t-s) r \left(\fint_{Q_{tr} } |F|^2 \right)^{1/2}
\end{equation}
for $(1/2) \le s< t\le  1$, where $C$ depends only on $d$.
\end{lemma}

\begin{proof}
The proof is  exactly the same as that for the usual Caccioppoli inequality.
Choose a cut-off function $\varphi \in C_0^\infty (Q_{tr})$ such that
$\varphi=1$ in $Q_{sr}$ and $|\nabla \varphi |\le C (t-s)^{-1}r^{-1} $.
Note that
$$
\aligned
\int_{Q_{tr}} \nabla u \cdot \nabla (u\varphi^2)\, dx
& =\int_{Q_{tr}  \cap \omega_{1, \eta}} \nabla u \cdot \nabla (u\varphi^2)\, dx\\
&=\int_{Q_{tr} \cap \omega_{1, \eta}} F (u\varphi^2)\, dx
=\int_{Q_{tr} } F (u\varphi^2)\, dx.
\endaligned
$$
Hence,
$$
\int_{Q_{tr}} |\nabla u|^2 \varphi^2\, dx
=- 2\int_{Q_{tr} } \varphi ( \nabla u \cdot \nabla \varphi) u\, dx+
\int_{Q_{tr} } F (u\varphi^2)\, dx,
$$
which, by the Cauchy inequality, yields \eqref{C}.
\end{proof}

\begin{proof}[\bf Proof of Theorem \ref{L-thm}]
By rescaling we may assume $\e=1$.
Let $u\in W^{1, 2} (Q_R)$ for some $R\ge 1$.
Suppose that $\Delta u=0$ in $Q_R\cap \omega_{1, \eta}$ and $u=0$ in $Q_R \setminus \omega_{1, \eta}$.
To prove \eqref{5.0-1},
without loss of generality,  we may assume $R\ge \delta^{-2}  $, where $\delta=\delta(d)>0$ is sufficiently small
(the case $1\le R \le \delta^{-2}$ is trivial).
Let $1\le r\le \delta R$.
By applying  the discrete Sobolev inequality \eqref{ds-1} to $g(z)= \hat{u}(z)$, we obtain
\begin{equation}\label{L-50}
\aligned
\max_{z\in \mathbb{Z}^d\cap  Q_{r+2}} |\hat{u}(z)|
 & \le \max_{z\in \mathbb{Z}^d \cap Q_{2\delta R} } |\hat{u}(z)|\\
 &  \le C \sum_{\ell=0}^N
 R^\ell \left(\frac{1}{R^d} \sum_{z\in \mathbb{Z}^d\cap Q_{4\delta R}} |\partial^\ell \hat{u}(z)|^2 \right)^{1/2}\\
 &\le C \left(\fint_{Q_{5\delta R}} |u|^2 \right)^{1/2}
 + C \sum_{\ell=1}^N
 R^\ell \left(\fint_{Q_{5\delta R }} |\nabla \partial^{\ell-1}  u|^2 \right)^{1/2},
 \endaligned
 \end{equation}
 where $N=[d/2] +1$ and  we have used \eqref{FTC-1} for the last inequality.
 
Next, let $v=\Delta^\gamma u$, where $\gamma $ is a multi-index with $1\le |\gamma |=\ell\le d$.
Observe that $\Delta v=0$ in $Q_{R-3 \ell}\cap \omega_{1, \eta} $ and $v=0$ in $Q_{R-3\ell}\setminus \omega_{1, \eta}$.
By applying Lemma \ref{C-lemma}  to $v$ we see that
$$
\aligned
\left(\fint_{Q_\rho} |\nabla \partial^\ell u|^2 \right)^{1/2}
& \le \frac{C}{\rho}
\left(\fint_{Q_{2\rho}}  |\partial^\ell u|^2 \right)^{1/2}\\
&\le  \frac{C}{\rho}
\left(\fint_{Q_{2\rho +3}} |\nabla \partial^{\ell-1} u|^2 \right)^{1/2}
\endaligned
$$
for $1\le \rho \le (R-3d )/2$, where, for the last step,  we have used the inequality 
\begin{equation}\label{FTC}
\left(\int_{y+Q_1} |\Delta_j u|^2 \right)^{1/2}
\le C \left(\int_{y+ 3Q_1} |\nabla u|^2 \right)^{1/2}.
\end{equation}
By induction it follows that 
\begin{equation}\label{C-ind}
\left(\fint_{Q_{\rho  }} |\nabla \partial^\ell u|^2 \right)^{1/2}
\le \frac{C}{\rho ^\ell}
\left(\fint_{Q_{C\rho} } |\nabla u|^2\right)^{1/2},
\end{equation}
where  $0\le \ell \le d$ and $C$ depends only on $d$.
This, together with Lemma \ref{dp-lemma} and \eqref{L-50}, shows that
for any $1\le r \le \delta R$,
$$
\aligned
\left(\fint_{Q_r} |u |^2 \right)^{1/2}
 & \le
C \left(\fint_{Q_R} |u|^2 \right)^{1/2}
+ C R  \left(\fint_{Q_{R/2}} |\nabla u|^2 \right)^{1/2}
+ C \left(\fint_{Q_{3r}} |\nabla u|^2\right)^{1/2}\\
& \le C \left(\fint_{Q_R} |u|^2 \right)^{1/2}
+ \frac{C}{r} \left(\fint_{Q_{4r}} |u|^2 \right)^{1/2}.
\endaligned
$$
It follows that
$$
\sup_{s\le r \le R}
\left(\fint_{Q_r} |u|^2 \right)^{1/2}
\le C \left(\fint_{Q_R} |u|^2 \right)^{1/2}
+ \frac{C}{s} \sup_{s\le r \le R} \left(\fint_{Q_r} |u|^2 \right)^{1/2},
$$
for any $1< s\le R$, where $C$ depends only on $d$.
By choosing $s=s(d)>1$ sufficiently  large we obtain
\begin{equation}\label{L-i-1}
\sup_{s\le r\le R} 
\left(\fint_{Q_r} |u|^2 \right)^{1/2}
\le C \left(\fint_{Q_R} |u|^2 \right)^{1/2}.
\end{equation}
Finally, we note that
$$
\sup_{1\le r \le  s} \left(\fint_{Q_r} |u|^2 \right)^{1/2}
\le C_s \left(\fint_{Q_s} |u|^2 \right)^{1/2}
\le C \left(\fint_{Q_R} |u|^2 \right)^{1/2},
$$
where we have used \eqref{L-i-1} for the last  step.
This, together with \eqref{L-i-1}, gives \eqref{5.0-1}.
\end{proof}



\section{Large-scale Lipschitz estimates}\label{section-Lip}

In this section we establish a large-scale Lipschitz estimate.
Recall that $Q_R =(-R/2, R/2)^d$.
As in the last section, we assume $\omega_{\e, \eta}$ is given by \eqref{omega}, where $T$ is the closure of an open subset of $Y$
with Lipschitz boundary.

\begin{thm}\label{Lip-thm}
Let $u\in W^{1, 2}(Q_R)$ for some $R\ge \e$.
Suppose that
\begin{equation}\label{5.0-0}
\Delta u =0 \quad \text{ in } \quad Q_R \cap \omega_{\e, \eta} \quad \text{ and } \quad
u=0 \quad \text{ in } Q_R \setminus  \omega_{\e, \eta}.
\end{equation}
Then, for $\e \le r\le  R$,
\begin{equation}\label{5.0-2}
\left(\fint_{Q_r} |\nabla u|^2 \right)^{1/2}
\le C \left(\fint_{Q_R} |\nabla u|^2 \right)^{1/2},
\end{equation}
where  $C$ depends on $d$ and $c_0$ in \eqref{condition-0}.
\end{thm}

We begin with a Poincar\'e inequality. 

\begin{lemma}\label{P-lemma}
Suppose that $u\in W^{1, 2}(Q_R)$ and $u=0$ on $Q_R\setminus \omega_{1, \eta} $, where  $R\ge 1$ is an odd integer.
Then 
\begin{equation}\label{P}
\int_{Q_R} |u|^2\, dx \le C \eta^{2-d} \int_{Q_R} |\nabla u|^2\, dx
\end{equation}
for $d\ge 3$, and
\begin{equation}\label{P-2}
\int_{Q_R} |u|^2\, dx
\le C |\ln (\eta/2) | \int_{Q_R} |\nabla u|^2\, dx
\end{equation}
for $d=2$, where $C$ depends only on $d$ and $c_0$.
\end{lemma}

\begin{proof}
The case $R=1$ is well known. See e.g. \cite[p.270]{Allaire-90} or \cite[Lemma 2.1]{Shen-2022} for a proof.
The general case follows  by covering $Q_R$ with unit cubes $\{ k + Q_1: k \in \mathbb{Z}^d\cap Q_R\}$.
\end{proof}

Following \cite{Jing-2020}, we introduce a $Y$-periodic function $\chi_\eta$ in $\R^d$ that satisfies 
\begin{equation}\label{corrector}
-\Delta \chi_\eta =\eta^{d-2} \quad \text{ in } \omega_{\e, \eta}  \quad \text{ and } \quad \chi_\eta =0 \quad \text{ in } \R^d \setminus \omega_{\e, \eta}.
\end{equation}
Let $H^1_{per}(Q_1)$ denote the closure in $H^1(Q_1)$  of the set of smooth $Y$-periodic functions in $\R^d$.
The existence and uniqueness of $\chi_\eta$ may be proved by using the Lax-Milgram Theorem on a Hilbert space
$H$, given by the closure of $\{ u \in H_{per}^1(Q_1): u=0 \text{ on } \eta T \}$ in $H^1_{per}(Q_1)$.

\begin{lemma}\label{c-lemma}
Let $\chi_\eta$ be given by \eqref{corrector}.
Then
\begin{equation}\label{c-10}
\aligned
C_1 \eta^{\frac{d-2}{2}}   \le &  \left(\fint_{Q_1} |\nabla \chi_\eta|^2 \right)^{1/2} \le C_2 \eta^{\frac{d-2}{2}},\\
C_1 \le  & \fint_{Q_1} \chi_\eta \le C_2,
\endaligned
\end{equation}
for $d\ge 3$, and
\begin{equation}\label{c-11}
\aligned
C_1 |\ln (\eta/2)|^{\frac12}  \le &  \left(\fint_{Q_1} |\nabla \chi_\eta|^2 \right)^{1/2} \le C_2  |\ln (\eta/2)|^{\frac12},\\
C_1 |\ln (\eta/2)| \le  & \fint_{Q_1} \chi_\eta \le C_2 |\ln (\eta/2)|,
\endaligned
\end{equation}
for $d=2$, where $C_1, C_2>0$ depend only on $d$ and $c_0$.
\end{lemma}

\begin{proof}
See \cite{Jing-2020} or  \cite[Lemma 4.4] {Shen-2022}.
\end{proof}

\begin{lemma}\label{d-S-lemma}
Let $u\in W^{1, 2}(Q_R)$ for some $R\ge 100d$.
Then, for any  $r\in [1, R/100]$,
\begin{equation}\label{dPS}
\left(\fint_{Q_{r}} | u - \hat{u} (0)|^2 \right)^{1/2}
\le C r \sum_{\ell=0}^N R^\ell \left(\fint_{Q_{R/2}} |\nabla \partial^\ell u|^2 \right)^{1/2}
+ C \left(\fint_{Q_{3r}} |\nabla u|^2 \right)^{1/2},
\end{equation}
where $N=[d/2]+1$ and $C$ depends only on $d$.
\end{lemma}

\begin{proof}
In view of Lemma \ref{dp-lemma} it suffices to show that if $z\in \mathbb{Z}^d\cap Q_{3r}$, then
$|\hat{u}(z) -\hat{u}(0)|$ is bounded by the first term in the right-hand side of \eqref{dPS}.
To this end  we use the observation,
$$
\max_{z\in \mathbb{Z}^d \cap  Q_{3r}} |\hat{u}(z)-\hat{u}(0)|
\le C r \max_{z\in \mathbb{Z}^d\cap Q_{3r}} |\partial \hat{u}(z)|.
$$
 By applying   the discrete Sobolev inequality \eqref{ds-1} with $g(z)=\partial \hat{u} (z)$,
 we obtain 
$$
\aligned
\max_{z\in \mathbb{Z}^d\cap Q_{3r}} |\hat{u}(z)-\hat{u}(0)|
 & \le C r  \sum_{\ell=0}^N R^\ell
\left(\frac{1}{R^d}
\sum_{z\in \mathbb{Z}^d \cap {Q}_{R/4}}
|\partial^{\ell +1} \hat{u} (z)  |^2 \right)^{1/2}\\
&
\le C r  \sum_{\ell=0}^N R^\ell
\left(\fint_{Q_{R/2}} |\nabla \partial^\ell u|^2 \right)^{1/2},
\endaligned
$$
where we have used \eqref{FTC-1} for the last inequality.
\end{proof}

We are now in a position to give the proof of Theorem \ref{Lip-thm}.

\begin{proof}[\bf Proof of Theorem \ref{Lip-thm}]
By rescaling we may assume $\e=1$.
Let $u\in W^{1, 2}(Q_R)$ be a solution of \eqref{5.0-0} for some $R\ge 1$.
To prove \eqref{5.0-2}, we may assume that $R\ge \delta^{-2}d$, where $\delta=\delta(d)>0$ is sufficiently small.
We may also assume that $R$ is an odd integer (the general case follows by choosing an odd integer $\ell $ so that
$R/2 \le \ell \le R$).

Let $w=u-\alpha \chi_\eta$, where $\chi_\eta$ is given by \eqref{corrector} and 
$\alpha\in \R$ is chosen so that $\hat{w}(0)=0$.
Since $\hat{u}(0)=\alpha \hat{\chi}_\eta(0)$, by Lemma \ref{c-lemma},  we have
\begin{equation}\label{L-100}
|\alpha  |
\le \left\{
\aligned
& C |\hat{u}(0)| & \quad & \text{ for } d\ge 3,\\
& C |\ln (\eta/2)|^{-1} |\hat{u}(0)| & \quad & \text{ for } d=2,
\endaligned
\right.
\end{equation}
where $C$ depends only on $d$ and $c_0$.
Let $r\in [1, \delta R]$.
By applying Lemma \ref{d-S-lemma} to $w$ we obtain 
\begin{equation}\label{L-99}
\left(\fint_{Q_r} |w|^2 \right)^{1/2}
\le C r \sum_{\ell=0}^N R^\ell \left(\fint_{Q_{100\delta R}} |\nabla \partial^\ell w|^2 \right)^{1/2}
+ C \left(\fint_{Q_{3r}} |\nabla w|^2 \right)^{1/2}.
\end{equation}

Next, we note that $-\Delta w= - \alpha \eta^{d-2}$ in $Q_R \cap \omega_{1, \eta}$ and
$w=0$ on $Q_R \setminus \omega_{1, \eta}$.
It follows by the Caccioppoli inequality in  Lemma \ref{C-lemma} that 
\begin{equation}\label{L-101}
\left(\fint_{Q_\rho} |\nabla w|^2 \right)^{1/2}
\le \frac{C}{\rho} \left(\fint_{Q_{2\rho}} |w|^2 \right)^{1/2}
+ C |\alpha| \eta^{d-2}  \rho
\end{equation}
for $1\le \rho \le R/2$.
Also, observe that if $\ell \ge 1$, then 
$\Delta (\partial^\ell w)=0$ in $Q_{R- 3 \ell} \cap \omega_{1, \eta}$ and $\partial^\ell w=0$ on $Q_{R-3\ell}  \setminus \omega_{1, \eta}$.
Hence, by the proof of \eqref{C-ind}, 
\begin{equation}\label{L-102}
\aligned
\left(\fint_{Q_\rho} |\nabla \partial^\ell w|^2 \right)^{1/2}
 & \le \frac{C}{\rho^\ell } \left(\fint_{Q_{C\rho}} |\nabla  w|^2 \right)^{1/2}\\
  \endaligned
\end{equation}
for $1\le \ell \le d$ and  $1\le \rho\le  \delta R$, where $C$ depends only on $d$.
It follows from \eqref{L-101} and \eqref{L-99} that for $1\le r\le \delta R/2$, 
$$
\aligned
\left(\fint_{Q_r} |\nabla w|^2 \right)^{1/2} 
& \le C \sum_{\ell=0}^N
R^\ell \left(\fint_{Q_{100\delta R}} |\nabla \partial^\ell w|^2\right)^{1/2}
+ \frac{C}{r}  \left(\fint_{Q_{6r}} |\nabla w|^2\right)^{1/2}
+ C |\alpha| \eta^{d-2} r\\
& \le C \left(\fint_{Q_{ R}} |\nabla w|^2 \right)^{1/2}
+ \frac{C}{r} \left(\fint_{Q_{6r}} |\nabla w|^2\right)^{1/2}
+ C |\alpha| \eta^{d-2} R,\\
\endaligned
$$
where $C$ depends only on $d$ and we have used \eqref{L-102} for the last inequality. Since $w=u-\alpha \chi_\eta$,
this yields, 
\begin{equation}\label{L-103}
\aligned
\left(\fint_{Q_r} |\nabla u|^2 \right)^{1/2}
&\le C \left(\fint_{Q_R} |\nabla u|^2 \right)^{1/2}
+\frac{C}{r}
\left(\fint_{Q_{6r}} |\nabla u|^2 \right)^{1/2}\\
&\qquad\qquad
+ C |\alpha | \left(\fint_{Q_1} |\nabla \chi_\eta|^2 \right)^{1/2}
+ C |\alpha | \eta^{d-2}  R,
\endaligned
\end{equation}
for any $1\le r\le R/6$ (the case $\delta R/2 \le r\le R/6$ is trivial).
We point out that the periodicity of $\nabla \chi_\eta$ is also used for \eqref{L-103}.

Suppose  $d\ge 3$. We consider two cases.
If $R\le \eta^{\frac{2-d}{2}}$, 
in view of \eqref{c-10} and \eqref{L-100}, we have 
$$
\aligned
|\alpha | \left(\fint_{Q_1} |\nabla \chi_\eta|^2 \right)^{1/2}
+ |\alpha | \eta^{d-2} R
&\le C \eta^{\frac{d-2}{2}} |\hat{u}(0)|
 \le C \eta^{\frac{d-2}{2}} \left(\fint_{Q_1} |u|^2 \right)^{1/2}\\
& \le C \eta^{\frac{d-2}{2}} \left(\fint_{Q_R} |u|^2 \right)^{1/2}\\
&\le C \left(\fint_{Q_R} |\nabla u|^2\right)^{1/2},
\endaligned
$$
where we have used the large-scale $L^\infty$ estimate \eqref{5.0-1} for the third inequality
and the Poincar\'e inequality \eqref{P} for the last step.
The assumption that $R$ is an odd integer is also used here.
This, together with \eqref{L-103}, gives
\begin{equation}\label{L-104}
\left(\fint_{Q_r} |\nabla u|^2 \right)^{1/2}
\le C \left(\fint_{Q_R} |\nabla u|^2 \right)^{1/2}
+\frac{C}{r} \left(\fint_{Q_{6r}} |\nabla u|^2 \right)^{1/2}
\end{equation}
for any $r\in [1, R/6]$, where $C$ depends only on $d$ and $c_0$. As in the proof of Theorem \ref{L-thm}, the large-scale Lipschitz estimate \eqref{5.0-2}
with $R\le \eta^{\frac{2-d}{2}}$
follows readily  from \eqref{L-104}.

Suppose $d\ge 3$ and $ R > \eta^{\frac{2-d}{2}}$. Let $\eta^{\frac{2-d}{2}}/2 \le r\le R/2$.
 We use the Caccioppoli inequality \eqref{C}  to obtain 
$$
\aligned
\left(\fint_{Q_r} |\nabla u|^2\right)^{1/2}
&\le \frac{C}{r} \left(\fint_{Q_{2r}} |u|^2 \right)^{1/2}
 \le C \eta^{\frac{d-2}{2}} \left(\fint_{Q_{2r}} |u|^2 \right)^{1/2}\\
& \le C \eta^{\frac{d-2}{2}} \left(\fint_{Q_R} |u|^2 \right)^{1/2}
 \le C \left(\fint_{Q_R} |\nabla u|^2 \right)^{1/2},
\endaligned
$$
where we have used the large-scale $L^\infty$ estimate \eqref{5.0-1} for the third inequality 
and the Poincar\'e inequality for the last step. 
As a result, we deduce that if $R> R_\eta =\eta^{\frac{2-d}{2}}$, 
$$
\aligned
\sup_{1\le r \le R}
\left(\fint_{Q_r} |\nabla u|^2 \right)^{1/2}
 & \le \sup_{1\le r \le  R_\eta}
\left(\fint_{Q_r} |\nabla u|^2 \right)^{1/2}
+ \sup_{R_\eta<  r \le R}
\left(\fint_{Q_r} |\nabla u|^2 \right)^{1/2}\\
&\le C  \sup_{R_\eta\le r \le R}
\left(\fint_{Q_r} |\nabla u|^2 \right)^{1/2}\\
&\le C \left(\fint_{Q_R} |\nabla u|^2 \right)^{1/2},
\endaligned
$$
where we have used the large-scale Lipschitz estimate for the case $R=R_\eta$ for the second
inequality.

The proof for the case $d=2$ is similar.
Again, we consider two cases.
 If $R\le  |\ln (\eta/2)|^{1/2}$, in view of \eqref{c-11} and \eqref{L-100},
 we have 
$$
\aligned
|\alpha | \left(\fint_{Q_1} |\nabla \chi_\eta|^2 \right)^{1/2}
+ |\alpha | \eta^{d-2} R
&\le C |\ln (\eta/2)|^{\frac12}  |\hat{u}(0)|
 \le C |\ln (\eta/2)|^{\frac12} \left(\fint_{Q_1} |u|^2 \right)^{1/2}\\
& \le C  |\ln (\eta/2)|^{\frac12} \left(\fint_{Q_R} |u|^2 \right)^{1/2}\\
&\le C \left(\fint_{Q_R} |\nabla u|^2\right)^{1/2},
\endaligned
$$
where we have used the large-scale $L^\infty$ estimate \eqref{5.0-1} for the third inequality and
 the Poincar\'e inequality \eqref{P-2} for the last step.
As in the case $d\ge 3$, this gives \eqref{L-104}, which leads to \eqref{5.0-2}.
If $R> R_\eta= |\ln (\eta/2)|^{1/2}$ and $R_\eta/2\le  r\le R/2$, 
we use the Caccioppoli inequality \eqref{C} and large-scale $L^\infty$ estimate \eqref{5.0-1}  to obtain 
$$
\aligned
\left(\fint_{Q_r} |\nabla u|^2\right)^{1/2}
&\le \frac{C}{r} \left(\fint_{Q_{2r}} |u|^2 \right)^{1/2}
 \le C |\ln (\eta/2)|^{1/2} \left(\fint_{Q_{2r}} |u|^2 \right)^{1/2}\\
& \le  C |\ln (\eta/2)|^{1/2}  \left(\fint_{Q_R} |u|^2 \right)^{1/2}
 \le C \left(\fint_{Q_R} |\nabla u|^2 \right)^{1/2},
\endaligned
$$
where we have used the 
Poincar\'e inequality \eqref{P-2}
 for the last step.
 This, combined  with the estimate for the case $R\le R_\eta$, yields \eqref{5.0-2} and completes the proof of Theorem \ref{Lip-thm}.
\end{proof}



\section{Large-scale $W^{1, p}$ estimates}\label{section-W}

Let $u\in W^{1, 2}(\R^d)$ and $u=0$ on $\R^d\setminus \omega_{\e, \eta}$.
Using Lemma \ref{P-lemma} and a rescaling argument, one may show that
\begin{equation}\label{P-00}
\int_{\omega_{\e, \eta}} |u|^2\, dx \le C \e^2 \eta^{2-d} \int_{\omega_{\e, \eta}} |\nabla u|^2\, dx
\end{equation}
for $d\ge 3$, and
\begin{equation}\label{P-02}
\int_{\omega_{\e, \eta}} |u|^2\, dx \le C \e^2  |\ln (\eta/2)| \int_{\omega_{\e, \eta}} |\nabla u|^2\, dx
\end{equation}
for $d=2$, where $\e, \eta \in (0, 1]$ and $C$ depends only on $d$ and $c_0$.
Let $W_0^{1, p}(\omega_{\e, \eta})$ denote the closure of $C_0^\infty(\omega_{\e, \eta})$ in $W^{1, p}(\omega_{\e, \eta})$.
It follows from \eqref{P-00}-\eqref{P-02}
by the Lax-Milgram Theorem  that for any $F\in L^2(\omega_{\e, \eta})$ and $f\in L^2 (\omega_{\e, \eta}; \R^d)$,  the Dirichlet problem 
\eqref{D-01} has a unique solution in $W_0^{1, 2}(\omega_{\e, \eta})$.
Moreover, the solution satisfies 
\begin{equation}\label{E}
\|\nabla u\|_{L^2(\omega_{\e, \eta})}
\le \| f \|_{L^2(\omega_{\e, \eta})} + C \e \eta^{1-\frac{d}{2}} \| F \|_{L^2(\omega_{\e, \eta})} 
\end{equation}
for $d\ge 3$, and
\begin{equation}\label{E-2}
\|\nabla u\|_{L^2(\omega_{\e, \eta})}
\le \| f \|_{L^2(\omega_{\e, \eta})} + C \e |\ln (\eta/2)|^{\frac12} \| F \|_{L^2(\omega_{\e, \eta})}
\end{equation}
for $d=2$.
The constants $C$ in \eqref{E}-\eqref{E-2} depend only on $d$ and $c_0$.

Let  $u\in W^{1, 2}_0(\omega_{\e, \eta}) $ be a weak solution of \eqref{D-01}. Define
\begin{equation}\label{S}
S_{\e, \eta} (F, f)(x) =\left(\fint_{x+\e Q_2} |\nabla u|^2 \right)^{1/2},
\end{equation}
where we have extended $u$ to $\R^d$ by zero.
It is easy to see that
\begin{equation}\label{S-1}
\| S_{\e, \eta} (F, f)\|_{L^2(\R^d)}= \|\nabla u\|_{L^2(\R^d)}.
\end{equation}
The following theorem gives the $L^p$ boundedness of $S_{\e, \eta} $  for $p\ge 2$.

\begin{thm}\label{S-thm}
Let $2\le p< \infty$ and  $\omega_{\e, \eta}$ be given by \eqref{omega}, where $T$ is the closure of an open subset of $Y$ with Lipschitz boundary.
Then, for any $f\in C_0^\infty (\R^d; \R^d)$ and $F\in C_0^\infty(\R^d)$,
\begin{equation}\label{S-01}
\| S_{\e, \eta} (F, f) \|_{L^p(\R^d)}
\le
\left\{
\aligned
& C  \| f \|_{L^p(\R^d)} +  C\e \eta^{1-\frac{d}{2}} \| F \|_{L^p(\R^d)}   & \quad & \text{ for } d\ge 3,\\
& C  \| f \|_{L^p(\R^d)} +  C\e  |\ln (\eta/2)|^{\frac12} \| F \|_{L^p(\R^d)}   & \quad & \text{ for } d=2,\\
\endaligned
\right.
\end{equation}
where $C$ depends on $d$, $p$ and $c_0$.
\end{thm}

The case $p=2$ follows readily from \eqref{S-1} and \eqref{E}-\eqref{E-2}.
To prove Theorem \ref{S-thm} for $p>2$, we use a real-variable argument and the large-scale Lipschitz estimate obtained in the last section.

An operator $\mathcal{S}$ is called sublinear if there exists a constant $K$ such that
\begin{equation}\label{K}
| \mathcal{S} (f+g)|
\le K \left\{ |\mathcal{S}(f)| + |\mathcal{S} (g)| \right\}.
\end{equation}

\begin{thm}\label{Shen-thm}
Let $\mathcal{S}$ be a bounded sublinear operator from $L^2(\R^d; \R^m)$ to $L^2(\R^d)$  with
$\| \mathcal{S} \|_{L^2\to L^2} \le C_0$.
Let $q>2$.
Suppose that 
\begin{equation}\label{Shen-1}
\left(\fint_B |\mathcal{S}(g)|^q \right)^{1/q}
\le N \left\{
\left(\fint_{2B} |\mathcal{S}(g)|^2 \right)^{1/2}
+ \sup_{B^\prime \supset B}
\left(\fint_B |g|^2 \right)^{1/2} 
\right\}
\end{equation}
for any ball $B$ in $\R^d$ and for any $g\in C_0^\infty(\R^d; \R^m)$ with supp$(g)\subset \R^d \setminus 4B$.
Then for any $f\in C_0^\infty (\R^d; \R^m)$,
\begin{equation}\label{Shen-2}
\| \mathcal{S} (f) \|_{L^p(\R^d)}
\le C_p \| f\|_{L^p(\R^d)},
\end{equation}
where $2< p< q$ and $C_p$ depends at most on $p$, $q$, $C_0$, $N$ and $K$ in \eqref{K}.
\end{thm}

\begin{proof}
See \cite{Shen-2005} or  \cite[pp.79-80]{Shen-book}.
\end{proof}

 Observe that 
by linearity,
\begin{equation}\label{L}
S_{\e, \eta}(F, f) \le S_{\e, \eta} (F, 0) + S_{\e, \eta} (0, f).
\end{equation}
We first treat the case $S_{\e, \eta}(0, f)$.

\begin{lemma}\label{S-lemma-1}
Let $2< p< \infty$ and $S_{\e, \eta}$ be defined by \eqref{S}. 
Then
\begin{equation}\label{S-11}
\| S_{\e, \eta} (0, f)\|_{L^p(\R^d)}
\le C \| f \|_{L^p(\R^d)}
\end{equation}
for any $f\in C_0^\infty(\R^d; \R^d)$, where $C$ depends only on $d$, $p$ and $c_0$ in \eqref{condition-0}.
\end{lemma}

\begin{proof}
By rescaling we may assume $\e=1$.
Let $\mathcal{S}(f) =S_{1, \eta} (0, f) $.
Note that  $\mathcal{S}$ satisfies \eqref{K} with $K=1$ and that
$\|\mathcal{S} \|_{L^2 \to L^2} \le 1$.
Let $Q$ be a cube in $\R^d$.
We will show that if $g\in C_0^\infty (\R^d; \R^d)$ with supp$(g)\subset \R^d \setminus 4Q$,
then
\begin{equation}\label{S-2}
\| \mathcal{S} (g)\|_{L^\infty(Q)}
\le C \left(\fint_{2Q} |\mathcal{S} (g)|^2\right)^{1/2},
\end{equation}
where $C$ depends only on $d$ and $c_0$.
By covering a ball $B=B(x_0, r)$ with non-overlapping cubes of side length $c_d r$,
it is not hard to deduce \eqref{Shen-1} from \eqref{S-2} (the second term in the right-hand side of \eqref{Shen-1} is not needed).
As a result, we obtain  \eqref{S-11} for any $f\in C_0^\infty(\R^d; \R^d)$.

Let $Q=Q(x_0, \ell)$ be a cube centered at $x_0$ and with side length $\ell$.
Suppose that $-\Delta u=\text{\rm div}(g)$ in $\omega_{1, \eta}$ and $u=0$ in $\R^d\setminus \omega_{1, \eta}$, where
$g \in C_0^\infty(\R^d; \R^d)$ and supp$(g) \subset \R^d \setminus 4Q$.
To show \eqref{S-2}, we use Theorem \ref{L-thm} as well as the observation, 
\begin{equation}\label{S-3}
 \left(\fint_{2Q} |\mathcal{S} (g)|^2\right)^{1/2}
 =\left(\frac{1}{(2\ell)^d}
 \int_{Q(x_0, 2 + 2 \ell)}
 |\nabla u(y)|^2 | Q(y, 2)\cap Q(x_0, 2\ell)|\, dy \right)^{1/2}.
\end{equation}
We consider two cases.
In the first case we assume $0< \ell\le 2$. Note that 
\begin{equation}\label{S-4}
\mathcal{S}(g)(x) \le 
\left( \int_{Q(x_0, 2+\ell)}
 |\nabla u(y) |^2 \, dy \right)^{1/2}
\end{equation}
for any $x\in Q(x_0, \ell)$.
Since $|Q(y, 2) \cap Q(x_0, 2\ell)|\ge c\, \ell^d $ for $y \in Q(x_0, 2+\ell)$,
we obtain \eqref{S-2} from \eqref{S-3} and \eqref{S-4}, with $C$ depending only on $d$.

In the second case we assume $\ell >2$.
Note that $\Delta u=0$ in $  \omega_{1, \eta}\cap Q(x_0, 4 \ell) $ and
$u=0$ in $\R^d\setminus  \omega_{1, \eta}$.
It follows by Theorem \ref{Lip-thm} that
$$
\left(\fint_{Q(x, 2)} |\nabla u|^2 \right)^{1/2}
\le C \left(\fint_{Q(x,\ell)} |\nabla u|^2 \right)^{1/2}
$$
for any $x\in Q(x_0, \ell)$, where $C$ depends only on $d$ and $c_0$.
Hence, for any $x\in Q(x_0, \ell)$,
$$
\aligned
\mathcal{S} (g) (x)
& \le C \left(\fint_{Q(x,\ell)} |\nabla u|^2 \right)^{1/2}\\
&\le 
C \left(\fint_{Q(x_0, 2\ell)} |\nabla u|^2 \right)^{1/2},
\endaligned
 $$
 where
 we have used the fact $Q(x, \ell)\subset Q(x_0, 2\ell)$ for $x\in Q(x_0, \ell)$.
 This shows that
 $$
 \aligned
 \| \mathcal{S}(g) \|_{L^\infty (Q)}
 &\le C \left(\fint_{Q(x_0, 2\ell)} |\nabla u|^2 \right)^{1/2}\\
 & \le C \left(\fint_{2Q} |\mathcal{S}(g)|^2 \right)^{1/2},
\endaligned
 $$
where, for the last inequality,  we have used \eqref{S-3} and the observation that
$|Q(y, 2) \cap Q(x_0, 2 \ell)| \ge c$ for any $y \in Q(x_0, 2\ell)$.
Consequently, we have proved \eqref{S-2} for any cube $Q$.
\end{proof}

Next, we deal with the operator $S_{\e, \eta}(F, 0)$.

\begin{lemma}\label{S-lemma-2}
Let $2<p<\infty$ and $S_{\e, \eta}$ be defined by \eqref{S}.
Then
\begin{equation}\label{S-21}
\| S_{\e, \eta}(F, 0) \|_{L^p(\omega_{\e, \eta})}
\le \left\{
\aligned
& C \e \eta^{1-\frac{d}{2}} \| F \|_{L^p(\omega_{\e, \eta})} & \quad & \text{ for } d\ge 3,\\
& C \e |\ln (\eta/2)|^{\frac12} \| F \|_{L^p(\omega_{\e, \eta})} & \quad & \text{ for } d=2,
\endaligned
\right.
\end{equation}
for any $F\in C_0^\infty(\R^d)$,
where $C$ depends only on $d$, $p$ and $c_0$.
\end{lemma}

\begin{proof}
As before, we may assume $\e=1$ by rescaling. Define
$$
\mathcal{S} (F)
=\left\{
\aligned
& \eta^{\frac{d}{2}-1} S_{1, \eta}(F, 0) & \quad & \text{ for } d\ge 3,\\
& |\ln (\eta/2)|^{-\frac12} S_{1, \eta} (F, 0) & \quad & \text{ for } d=2.
\endaligned
\right.
$$
Then $\mathcal{S}$ satisfies \eqref{K} with $K=1$, and  $\| \mathcal{S}\|_{L^2\to L^2} \le C_0$ by \eqref{E}-\eqref{E-2}.
Let $u$ be a weak solution of $-\Delta u=G$ in $\omega_{1, \eta}$ with $u=0$ on $\partial \omega_{1, \eta}$,
where $G\in C_0^\infty(\R^d)$ and supp$(G) \subset \R^d\setminus 4Q$.
Since $\Delta u=0$ in $\omega_{1, \eta}\cap 4Q$ and $u=0$ in  $\R^d\setminus \omega_{1, \eta}$,  the same argument as in the proof of Lemma \ref{S-lemma-1}
yields the estimate \eqref{S-2}.
As a result, by Theorem \ref{Shen-thm}, we obtain 
$$
\| \mathcal{S} (F)\|_{L^p(\omega_{1, \eta})}
\le C \| F \|_{L^p(\omega_{1, \eta})}
$$
for any $2<p<\infty$, 
where $C$ depends only on $d$, $p$ and $c_0$.
This gives \eqref{S-21} with $\e=1$.
\end{proof}

\begin{proof}[\bf Proof of Theorem \ref{S-thm}]
In view of \eqref{L}, the estimates in \eqref{S-01}
follow readily  from \eqref{S-11} and \eqref{S-21}.
\end{proof}



\section{Estimates in an exterior domain}\label{section-loc1}

In this section we establish $W^{1, p}$ estimates for solutions with compact support  of Laplace's equation in the 
exterior domain $\R^d\setminus T$, where $T$ is the closure of a bounded $C^1$ domain in $\R^d$ with connected 
boundary. We  assume that $B(0, c_0) \subset T$.

 We begin with  $W^{1, p}$ estimates for Laplace's equation in a bounded Lipschitz or $C^1$ domain.
 
\begin{thm}\label{JK-lemma}
Let $\Omega$ be a bounded Lipschitz domain in $\R^d$. There exists $\delta>0$, depending on $d$ and the Lipschitz character of $\Omega$, such that
if 
\begin{equation}\label{Lip-R}
\Big |\frac{1}{p}-\frac12 \Big |
< 
\left\{
\aligned
& \frac16 +\delta & \quad & \text{ for } d\ge 3,\\
& \frac14  +\delta & \quad & \text{ for } d=2,
\endaligned
\right.
\end{equation}
 the Dirichlet problem, 
$-\Delta u=F$ in $\Omega$ and $u=0$ on $\partial\Omega$, has a unique solution in $W_0^{1,p}(\Omega)$
for any $F\in W^{-1, p}(\Omega)$.
Moreover, the solution satisfies the estimate,
\begin{equation}\label{JK-e}
\| \nabla u \|_{L^p(\Omega)} \le C_p \| F \|_{W^{-1, p}(\Omega)},
\end{equation}
where $C_p$ depends on $d$, $p$ and the Lipschitz character of $\Omega$.
Furthermore, if $\Omega$ is a bounded $C^1$ domain, the results above hold for $1< p< \infty$.
\end{thm}

\begin{proof}
The estimate \eqref{JK-e}  for $1< p< \infty$ is  well known if $\Omega$ is a $C^{1, \alpha}$ domain.
For Lipchitz  and $C^1$ domains,  the theorem was proved in  \cite{JK}.
\end{proof}

The next theorem is on the solvability  of the Dirichlet problem in a weighted Sobolev space in the exterior domain $\R^d\setminus T$,
\begin{equation}\label{ext-P}
-\Delta u= F \quad  \text{ in } \R^d\setminus T\quad  \text{ and } \quad u=0 \quad \text{ on } \partial T.
\end{equation}
We first introduce some notations.
For $1<p< \infty$ and $p\neq d$, let
\begin{equation}\label{X-p}
\aligned
X^{1, p} (\R^d\setminus T)
=  \Big\{ u\in W^{1, p}_{loc}(\R^d\setminus T): \ \
& (1+ |x|)^{-1} u \in L^p(\R^d\setminus T) \\
& \text{ and } 
\nabla u  \in L^p(\R^d\setminus T) \Big\},
 \endaligned
\end{equation}
with its natural norm,
\begin{equation}\label{X-norm}
\| u \|_{X^{1, p} (\R^d\setminus T)}
= \| (1+|x|)^{-1}  u \|_{L^p(\R^d\setminus T)}
+ \|\nabla u \|_{L^p(\R^d\setminus T)}.
\end{equation}
If $p=d$, let 
\begin{equation}\label{X-d}
\aligned
X^{1, d} (\R^d\setminus T)
=  \Big\{ u\in W^{1, d}_{loc}(\R^d\setminus T): \ \  & 
((1+|x|) \ln (2+|x|) )^{-1} u \in L^d(\R^d\setminus T)  \\
& \text{ and } 
\nabla u \in L^d(\R^d\setminus T)\Big\},
 \endaligned
\end{equation}
and
\begin{equation}\label{X-d-norm}
\| u \|_{X^{1, d} (\R^d\setminus T)}
= \| (( 1+|x|)\ln (2+|x|)) ^{-1}  u \|_{L^d(\R^d\setminus T)}
+ \|\nabla u \|_{L^d(\R^d\setminus T)}.
\end{equation}
It  follows from \cite[Theorem 1.1]{AGG-1997}  that  for $u \in X^{1, p}(\R^d\setminus T)$,
\begin{equation}\label{P-X}
\aligned
\| u \|_{X^{1, p}(\R^d\setminus T)}
 & \le C \| \nabla u \|_{L^p(\R^d\setminus T)} & \quad & \text{ if } 1<p< d,\\
 \inf_{\alpha \in \R} 
 \| u-\alpha \|_{X^{1, p} (\R^d\setminus T)}
 & \le C \|\nabla u \|_{L^p(\R^d\setminus T)}  & \quad & \text{ if  } d\le p< \infty.
 \endaligned
 \end{equation}

Let 
\begin{equation}\label{X-0}
X_0^{1, p} (\R^d\setminus T)
=\left\{  u \in X^{1, p}(\R^d\setminus T): \ u=0 \text{ on } \partial T \right\},
\end{equation}
and $X^{-1, p} (\R^d\setminus T)$ be the dual of $X^{1, p^\prime}_0(\R^d\setminus T)$, where $p^\prime =\frac{p}{p-1}$.
It is known that $C_0^\infty(\R^d)$ is dense in $X^{1, p}(\R^d\setminus T)$ and 
$C_0^\infty(\R^d\setminus T)$ is dense in $X_0^{1, p}(\R^d\setminus T)$ \cite{AGG-1997}.

Let
\begin{equation}\label{V-p}
V^p_0(\R^d\setminus T)
=\left\{ w \in X^{1, p}_0(\R^d\setminus T): \ \Delta w=0 \text{ in } \R^d\setminus T \right\}.
\end{equation}

\begin{thm}\label{ext-thm}
Let $d\ge 2$ and $2\le p< \infty$.
Let $T$ be the closure of a bounded $C^1$ domain in $\R^d$ with connected boundary.
Then, for any $F\in X^{-1, p}(\R^d\setminus T)$, the Dirichlet problem \eqref{ext-P} 
has a unique solution in $X_0^{1, p}(\R^d\setminus T)/V_0^p(\R^d\setminus T)$.
Moreover, the solution satisfies 
\begin{equation}\label{ext-1}
\inf_{w\in V^p_0(\R^d\setminus T)} 
\| u-w \|_{X^{1, p} (\R^d\setminus T)}
\le C \| F \|_{X^{-1, p} (\R^d\setminus T)},
\end{equation}
where $C$ depends on $d$, $p$ and $T$.
\end{thm}

\begin{proof}
This was proved in \cite[Theorem 2.10]{AGG-1997}  under the assumption that  $\partial T$ is $C^{1, 1}$.
With the $W^{1, p}$ estimates in Theorem \ref{JK-lemma} for bounded  $C^1$ domains, an inspection of the proof shows  that 
Theorem 2.10 in \cite{AGG-1997} continues to hold under the condition that $\partial T$ is $C^1$.
\end{proof}

A few remarks are in order.

\begin{remark}\label{ext-r-1}
{\rm
If  $d\ge 3$ and $2\le  p< d$, or $d=p=2$, then
\begin{equation}\label{V-p-1}
V_0^p(\R^d\setminus T) =\{ 0\}.
\end{equation}
As a result, the solution of \eqref{ext-P} is unique in $X_0^{1, p}(\R^d\setminus T)$ and satisfies 
\begin{equation}\label{ext-2}
\| u \|_{X^{1, p}(\R^d\setminus T)} \le C \| F \|_{X^{-1, p} (\R^d\setminus T)}.
\end{equation}
}
\end{remark}

\begin{remark}\label{ext-r-2}
{\rm
Suppose  $d\ge 3$ and $p\ge d$, Then 
\begin{equation}\label{V-p-2}
V_0^p (\R^d\setminus T)
= \left\{ \alpha \phi_*: \ \alpha \in \R\right\},
\end{equation}
where $\phi_*$ is  the unique solution of the exterior problem,
\begin{equation}\label{e-p-1}
\left\{
\aligned
\Delta \phi_* & =0 & \quad & \text{ in } \R^d\setminus {T},\\
\phi_*&=0& \quad & \text{ on } \partial T,\\
\phi_* (x)& \to 1 & \quad & \text{ as } |x| \to \infty.
\endaligned
\right.
\end{equation}
Moreover,   the solution is given by
$$
\phi_* (x) =1-\int_{\partial T} \frac{g_* (y)}{|x-y|^{d-2}} \, d\sigma (y)
$$
for some $g_* \in L^2(\partial T)$ \cite{Verchota-1982, AGG-1997}. It follows that if $0\in T$,
\begin{equation}\label{e-p-2}
\left\{
\aligned
  \phi_* (x) & =1- c_* |x|^{2-d} + O(|x|^{1-d}), \\
  \nabla \phi_* (x)  & =-c_* \nabla (|x|^{2-d}) + O(|x|^{-d}),\\
  \nabla^2 \phi_* (x)  & = O(|x|^{-d}),
\endaligned
\right.
\end{equation}
as $|x| \to \infty$,
where
$$
c_* =\int_{\partial T} g_* (y)\, d\sigma (y) \neq 0.
$$
}
\end{remark}

\begin{remark}\label{ext-r-3}
{\rm
If $d=2$ and $p>2$, then 
\begin{equation}\label{V-p-3}
V_0^p (\R^d\setminus T)
= \left\{ \alpha \phi_*: \ \alpha \in \R\right\},
\end{equation}
where $\phi_*$ is  a harmonic function in $\R^2\setminus {T}$ with the properties that  $\phi_*=0$ on $\partial T$ and
\begin{equation}\label{e-p-3}
\left\{
\aligned
\phi_* (x)  & =-c_* \ln |x| + O(|x|^{-1}),\\
\nabla \phi_* (x) & = -c_* \nabla (\ln |x|) + O(|x|^{-2}) ,\\
\nabla^2 \phi_* (x)  & = O(|x|^{-2}),
\endaligned
\right.
\end{equation}
as $|x| \to \infty$ \cite{Verchota-1982, AGG-1997}.
}
\end{remark}

\begin{thm}\label{thm-ext1}
Let $d\ge 2$ and $2< p< \infty$.
Let $u\in W^{1, p}(\R^d\setminus T)$ be a solution of
\begin{equation}\label{ext-DP}
-\Delta u = F +\text{\rm div}(f) \quad \text{  in } \R^d\setminus T
\quad \text{ and } \quad u=0 \quad \text{  on }\partial T.
\end{equation}
Suppose that $T\subset B(0, R)$ and
supp$(u)$, supp$(F)$, supp$(f)$ $\subset B(0, R)$ for some $R\ge 2$.
Then
\begin{equation}\label{ext-10}
\| \nabla u \|_{L^p(\R^d\setminus T)}
\le C \Phi_p(R)
\left\{ \| f\|_{L^p(\R^d\setminus T)}
+ R \| F \|_{L^p(\R^d\setminus T)} \right\},
\end{equation}
where 
\begin{equation}\label{Phi}
\Phi _p(R)=
\left\{
\aligned
& 1 & \quad & \text{ if } d\ge 3 \text{ and } 2<p< d,\\
& (\ln R)^{1-\frac{1}{d}} & \quad & \text{ if } d\ge 3 \text{ and } p=d,\\
& R^{1-\frac{d}{p} } & \quad  & \text{ if } d\ge 3 \text{ and } d< p< \infty,\\
& R^{1-\frac{2}{p}} (\ln R)^{-1} & \quad & \text{ if } d=2 \text{ and } 2< p< \infty,
\endaligned
\right.
\end{equation}
and $C$ depends only on $d$, $p$ and $T$.
\end{thm}

\begin{proof}
Note that $W^{1, p} (\R^d\setminus T)\subset X^{1, p} (\R^d\setminus T)$, and that  for any $\psi \in X^{1, p^\prime}_0(\R^d\setminus T)$,
$$
\aligned
\Big |\int_{\R^d\setminus T} F \psi\, dx \Big|
 & \le \| F \|_{L^p(B(0, R))} \| \psi \|_{L^{p^\prime} (B(0, R))}\\
& \le  CR   \| F \|_{L^p(B(0, R))} \| \psi \|_{X^{1, p^\prime}(\R^d\setminus T)},
\endaligned$$
where we have used the facts that supp$(F)\subset B(0, R)$ and $p^\prime \neq d$.
It follows that 
\begin{equation}
\aligned
\| F +\text{\rm div}(f)\|_{X^{-1, p}(\R^d\setminus T)}
\le  C \left\{ \| f\|_{L^p(\R^d\setminus T)}
+ R \| F \|_{L^p(\R^d\setminus T)} \right\}.
\endaligned
\end{equation}
This allows us to apply Theorem \ref{ext-thm} to obtain 
\begin{equation}\label{ext-11}
\inf_{w\in V_0^p(\R^d\setminus T)} \| u - w \|_{X^{1, p}(\R^d\setminus T)} 
\le C \left\{ \| f\|_{L^p(\R^d\setminus T)}
+ R \| F \|_{L^p(\R^d\setminus T)} \right\}.
\end{equation}

Suppose $d\ge 3$ and $2<p<d$.
Then $V_0^p(\R^d\setminus T)= \{ 0\}$. 
It follows from \eqref{ext-11}  that 
\begin{equation}\label{ext-12}
\| \nabla u \|_{L^p(\R^d\setminus T)}
\le C 
\left\{ \| f\|_{L^p(\R^d\setminus T)}
+ R \| F \|_{L^p(\R^d\setminus T)} \right\}.
\end{equation}
Let $d\ge 3$ and  $d< p< \infty$. Then, by Remark \ref{ext-r-2}, 
$V_0^p(\R^d\setminus T) =\{ \alpha \phi_*: \alpha \in \R\}$,
where the harmonic function  $\phi_*$ satisfies \eqref{e-p-2}.
Let
\begin{equation}\label{ext-13}
\inf_{w\in V_0^p(\R^d\setminus T)} \| u -w \|_{X^{1, p} (\R^d\setminus T)} 
=\| u -\alpha_0 \phi_*\|_{X^{1, p} (\R^d\setminus T)}
\end{equation}
for some $\alpha_0\in \R$.
Since $u=0$ in $\R^d\setminus B(0, R)$, it follows by \eqref{ext-11} that
$$
|\alpha_0 | \| |x|^{-1} \phi_* \|_{L^p(\R^d\setminus B(0, R))}
\le C 
\left\{ \| f\|_{L^p(\R^d\setminus T)}
+ R \| F \|_{L^p(\R^d\setminus T)} \right\}.
$$
Since $\phi_*\sim 1$ for $|x|$ large, this yields 
$$
|\alpha_0|\le C R^{1-\frac{d}{p}} \left\{ \| f\|_{L^p(\R^d\setminus T)}
+ R \| F \|_{L^p(\R^d\setminus T)} \right\}.
$$
Hence,
\begin{equation}\label{ext-14}
\aligned
\| \nabla u\|_{L^p(B(0, R)\setminus T)}
 & \le \| \nabla (u-\alpha_0 \phi_*) \|_{L^p(\R^d\setminus T)}
+ |\alpha_0| \| \nabla \phi_* \|_{L^p(\R^d\setminus T)}\\
& \le C \left\{ \| f\|_{L^p(\R^d\setminus T)}
+ R \| F \|_{L^p(\R^d\setminus T)} \right\} + C |\alpha_0| \\
& \le  C R^{1-\frac{d}{p}}\left\{ \| f\|_{L^p(\R^d\setminus T)}
+ R \| F \|_{L^p(\R^d\setminus T)} \right\}.
\endaligned
\end{equation}
If $d\ge 3$ and $p=d$, a similar argument shows that
$$
|\alpha_0| \le C ( \ln  R)^{1-\frac{1}{d}} \left\{ \| f\|_{L^p(\R^d\setminus T)}
+ R \| F \|_{L^p(\R^d\setminus T)} \right\},
$$
and
\begin{equation}\label{ext-15}
\| \nabla u\|_{L^p(B(0, R)\setminus T)}
\le 
 C ( \ln  R)^{1-\frac{1}{d}} \left\{ \| f\|_{L^p(\R^d\setminus T)}
+ R \| F \|_{L^p(\R^d\setminus T)} \right\}.
\end{equation}

The argument above works equally well for $d=2$ and $2<p< \infty$.
In this case, using \eqref{e-p-3}, we obtain 
$$
|\alpha_0|\le C R^{1-\frac{2}{p}}( \ln R)^{-1}  \left\{ \| f\|_{L^p(\R^d\setminus T)}
+ R \| F \|_{L^p(\R^d\setminus T)} \right\},
$$
and
\begin{equation}\label{ext-16}
\| \nabla u\|_{L^p(B(0, R)\setminus T)}
\le C R^{1-\frac{2}{p}}( \ln R)^{-1}   \left\{ \| f\|_{L^p(\R^d\setminus T)}
+ R \| F \|_{L^p(\R^d\setminus T)} \right\}.
\end{equation}
This completes the proof.
\end{proof}


\begin{cor}\label{cor-l-00}
Let $d\ge 2$ and $2<p< \infty$.
Let $u$ be a solution of $-\Delta u=F+\text{\rm div}(f)$ in $R \tY \setminus T$ with
$u=0$ on $\partial  T$, where $\tY= (1+c_0)Q_1$.
Then,  for $R\ge 3$, 
\begin{equation}\label{l-100}
\| \nabla u\|_{L^p(Q_R \setminus T)}
\le C \Phi_p (R)  \left\{
\| f \|_{L^p(R\tY\setminus T)}
+ R\| F \|_{L^p(R\tY\setminus T)} + R^{\frac{d}{p} -\frac{d}{2} -1} \| u \|_{L^2(R\tY \setminus B(0, R/3))} \right\},
\end{equation}
where  $\Phi_p(R)$ is given by \eqref{Phi} and $C$ depends only on $d$, $p$ and $T$.
\end{cor}

\begin{proof}
Choose a cut-off function $\varphi \in C_0^\infty((1+c_0/3) Q_R )$ such that $\varphi=1$ in $Q_R$ and
$|\nabla \varphi|\le C R^{-1}$, $|\nabla^2 \varphi | \le C R^{-2}$.
Note that $u \varphi=0$ on $\partial T$ and
$$
-\Delta (u\varphi)
=F \varphi + \text{\rm div}( f \varphi )
- f\cdot \nabla \varphi
-2 \text{\rm div} ( u \nabla \varphi)
+ u\Delta \varphi
$$
in $\R^d\setminus T$.
It follows by Theorem \ref{thm-ext1} that 
\begin{equation}\label{l-101}
\aligned
\|\nabla u \|_{L^p(Q_R\setminus T)}
&\le \| \nabla (u\varphi)\|_{L^p(\R^d\setminus T)}\\
& \le C \Phi_p(R)
 \left\{  \| f \|_{L^p(R\tY\setminus T)}
+ R\| F \|_{L^p(R\tY\setminus T)} + R^{-1} \| u \|_{L^p((1+c_0/3)Q_R \setminus Q_R)} \right\},
\endaligned
\end{equation}
where $\Phi_p(R)$ is given by \eqref{Phi}.
Using interior estimates for Laplace's equation, one may show that
$$
\| u\|_{L^p((1+c_0/3)Q_R \setminus Q_R)}
\le CR^{\frac{d}{p}-\frac{d}{2}} \| u\|_{L^2(R\tY\setminus B(0, R/3) )}
+ CR \| f\|_{L^p(\tY\setminus T)}
+ C R^2 \| F \|_{L^p(\tY \setminus T)},
$$
which, together with \eqref{l-101}, yields \eqref{l-100}.
\end{proof}



\section{Local estimates in a cell}\label{section-loc2}

In this section we establish $W^{1, p}$  estimates for solutions of
\begin{equation}\label{l-0}
\left\{
\aligned
-\Delta u & = F +\text{\rm div}(f) & \quad & \text{ in } \tY \setminus \eta T,\\
u&=0 & \quad & \text{ on } \partial (\eta T),
\endaligned
\right.
\end{equation}
where $\tY = (1+c_0)Q_1$ and $\eta \in (0, (4d)^{-1})$.
Throughout the section, unless indicated otherwise, we assume  that $T$  is the closure of a bounded $C^1$ subdomain of $Y$
and satisfies \eqref{condition-0}.
Let $\Phi_p (R)$ be given by \eqref{Phi}.
Our goal is to prove the following.

\begin{thm}\label{thm-local}
Let $2< p< \infty$. Suppose that
 $u$ is a solution of \eqref{l-0} with $F\in L^p(\tY\setminus \eta T)$ and $f\in L^p(\tY\setminus \eta T; \R^d)$.
 Let  $\alpha \in \R$.
Then, for $d\ge 3$,
\begin{equation}\label{l-1}
\aligned
\| \nabla u \|_{L^p(Y\setminus \eta T)}
 & \le C  |\alpha| \eta^{\frac{d}{p}-1}
+C \Phi_p (\eta^{-1})
 \left(  \int_{\tY \setminus \eta T} \left(  |f|^p+ |F|^p \right)\, dx\right)^{1/p}\\
&\qquad
+C \Phi_p (\eta^{-1})  \left(\int_{\tY \setminus B(0, 1/3)} |u-\alpha |^2 \, dx \right)^{1/2},
\endaligned
\end{equation}
and for $d=2$,
\begin{equation}\label{l-1-2d}
\aligned
\| \nabla u \|_{L^p(Y\setminus \eta T)}
 & \le C  |\alpha| \eta^{\frac{2}{p}-1} |\ln \eta|^{-1}
+C\eta^{\frac{2}{p}-1} |\ln \eta|^{-1}
 \left(  \int_{\tY \setminus \eta T} \left(  |f|^p+ |F|^p \right)\, dx\right)^{1/p}\\
&\qquad
+C\eta^{\frac{2}{p}-1} |\ln \eta|^{-1} \left(\int_{\tY \setminus B(0, 1/3)} |u-\alpha |^2 \, dx \right)^{1/2},
\endaligned
\end{equation}
where  $C$ depends only on $d$, $p$ and $T$.
\end{thm}

\begin{lemma}\label{lemma-l-1}
Let $2<p<\infty$.
Let $u$ be the same as in Theorem \ref{thm-local}. Then
\begin{equation}\label{local-ext}
\| \nabla u \|_{L^p(Y\setminus \eta T)}
\le C \Phi_p (\eta^{-1})\left\{ 
\| u\|_{L^2(\tY \setminus B(0, 1/3))}+
 \| f\|_{L^p(\tY\setminus \eta T)}
+ \| F \|_{L^p(\tY\setminus \eta T)} \right\},
\end{equation}
where  $\Phi_p$ is give by \eqref{Phi} and $C$ depends only on $d$, $p$ and $T$.
\end{lemma}

\begin{proof}
This follows readily from Corollary  \ref{cor-l-00} by a simple rescaling argument.
Indeed, suppose $-\Delta u=F +\text{\rm div}(f)$ in $\tY\setminus \eta T$.
Let $v(x)=u(\eta x)$. Then
$
-\Delta v =G +\text{\rm div}(g)$ in $R \tY\setminus T$, where $R=\eta^{-1}$,
$G(x) =\eta^2 F(\eta x)$ and $g (x) =\eta f(\eta x)$.
\end{proof}

Note that if $u$ is a solution of \eqref{l-0} and $\alpha \neq 0$, then $u-\alpha$ is not a solution of \eqref{l-0} since it 
does not satisfy the boundary condition on $\partial (\eta T)$.
To prove Theorem \ref{thm-local}, 
we need to construct  a corrector $\psi_\eta$ such that $\psi_\eta=0$ on $\partial T$ and
$\psi_\eta = 1$ on $(1+c_0)Y \setminus B(0, 1/3)$.

Let $d\ge 3$. 
Let $\phi_*$ be defined by \eqref{e-p-1}.
For each $\eta \in (0, 1/(4d))$, we introduce a function $\psi_\eta$ in $Y$, defined by
\begin{equation}\label{psi}
\psi_\eta (x)
=\left\{
\aligned
& 1 & \quad & \text{ if } x\in Y \setminus B(0, 1/3),\\
& \phi_* (x/\eta) &\quad  & \text{ if } x\in B(0, 1/4) \setminus \eta T ,\\
& 0 & \quad &  \text{ if } x\in \eta T,
\endaligned
\right.
\end{equation}
and $\psi_\eta$ is the harmonic function in $B(0, 1/3) \setminus \overline{B(0, 1/4)}$ such that $\psi_\eta =1$ on $\partial B(0, 1/3)$ and
$\psi_\eta(x) =\phi_* (x/\eta)$ on $\partial B(0, 1/4)$.
In the case $d=2$, we define $\psi_\eta$ by
\begin{equation}\label{2d-p}
\psi_\eta (x)=
\left\{
\aligned
& 1 &\quad & \text{ if } x\in Y\setminus B(0, 1/3), \\
&\frac{\ln |x| -\ln (d\eta)}{ \ln (1/3)-\ln (d\eta)}  & \quad & \text{ if } x \in B(0, 1/3) \setminus B(0, d\eta),\\
& 0 & \quad & \text{ if } x\in B(0, d \eta).
\endaligned
\right.
\end{equation}
 Since $\psi_\eta =1$  on $\partial Y $, we may extend $\psi_\eta$ to $\R^d$ periodically.
 Thus, $\psi_\eta$ is $Y$-periodic, i.e., $\psi_\eta (x+k) =\psi_\eta (x)$ for any $x\in \R^d$ and  $k \in \mathbb{Z}^d$.
Note that $0 \le \psi_\eta\le 1$ for $d=2$.
By the maximum principle,  the  same is true  for $d\ge 3$.

\begin{lemma}\label{p0-lemma}
Let $\psi_\eta$ be defined by \eqref{psi}-\eqref{2d-p}.
If $d\ge 3$,
\begin{equation}\label{p0-0}
\left(\int_{Y} |\nabla \psi_\eta|^p \, dx \right)^{1/p}
\approx 
\left\{
\aligned
&\eta^{\frac{d}{p}-1}   & \quad & \text{ if } d^\prime<p<\infty, \\
&  \eta^{d-2} |\ln \eta|^{\frac{1}{p} } & \quad & \text{ if } p=d^\prime,\\
& \eta^{d-2}  & \quad & \text{ if } 1< p< d^\prime,
\endaligned
\right.
\end{equation}
where $d^\prime =\frac{d}{d-1}$.
If $d=2$, we have $\| \nabla \psi_\eta\|_{L^p(Y)} \approx \eta^{\frac{2}{p}-1} |\ln \eta|^{-1}$ for $2< p< \infty$,
$\|\nabla \psi_\eta\|_{L^p(Y)} \approx |\ln \eta|^{-1/2}$ for $p=2$, 
and $\| \nabla \psi_\eta \|_{L^p(Y)} \approx |\ln \eta|^{-1}$ for $1< p< 2$.
\end{lemma}

\begin{proof} 
The case $d=2$ follows by a direct calculation.
Consider the case $d\ge 3$.
Since $\psi_\eta (x)=\phi_* (x/\eta)$ in $B(0, 1/4)\setminus \eta T$, we have
\begin{equation}\label{p0-1}
\aligned
\int_{B(0, 1/4)\setminus \eta T} |\nabla \psi_\eta|^p\, dx
 & =\eta^{d-p} \int_{B(0, (4\eta)^{-1})\setminus T} |\nabla \phi_*|^p\, dx\\
 &\approx 
 \left\{
 \aligned
 & \eta^{d-p} & \quad & \text{ if } p>d^\prime ,\\
 & \eta^{d-p} |\ln \eta| & \quad & \text{ if } p=d^\prime,\\
 & \eta^{(d-2)p} & \quad & \text{ if } 1<p< d^\prime,
 \endaligned
 \right.
 \endaligned
 \end{equation}
where we have used \eqref{e-p-2}. We also used the fact that 
$|\nabla \phi_*| \in L^p(2T\setminus T)$ for any $1<p<\infty$, under the assumption that $\partial T$ is $C^1$.

To bound $\nabla \psi_\eta$ on $B(0, 1/3)\setminus B(0, 1/4)$, we  observe  that
$w=\psi_\eta-1$ is harmonic in $B(0, 1/3)\setminus B(0, 1/4)$ and  $w=0$ on $\partial B(0, 1/3)$,
$w=\phi_*(x/\eta)-1$ on $\partial B(0, 1/4)$.
By \eqref{e-p-2} and regularity estimates for harmonic functions,
we obtain $|\nabla \psi_\eta| =|\nabla w| \le C  \eta^{d-2}$ in $B(0, 1/3)\setminus B(0, 1/4)$.
This, together with \eqref{p0-1}, gives \eqref{p0-0}.
\end{proof}

\begin{lemma}\label{p-lemma}
Let $\psi_\eta$ be defined by \eqref{psi}-\eqref{2d-p} and extended periodically to $\R^d$.
Then
\begin{equation}\label{p-1}
\left\{
\aligned
-\Delta \psi_\eta  & = F_\eta +\text{\rm div} (f_\eta) & \quad & \text{ in } \omega_{1, \eta}, \\
\psi_\eta & =0 & \quad & \text{ in } \R^d\setminus  \omega_{1, \eta},
\endaligned
\right.
\end{equation}
where $F_\eta$ and $f_\eta$ are $Y$-periodic functions satisfying 
\begin{equation} \label{p-2}
\aligned
| F_\eta| + |f_\eta|  & \le C \eta^{d-2}
 \endaligned
\qquad
 \text{ in } Y\setminus \eta T
\end{equation}
for $d\ge 3$,
and
\begin{equation} \label{p-3}
|F_\eta| + |f_\eta| \le C |\ln \eta|^{-1}
\qquad \text{ in }Y \setminus \eta T
\end{equation}
for $d=2$. The constant $C$ depends only on $d$ and $T$.
\end{lemma}

\begin{proof}
We first  consider the case $d\ge 3$.
Let $\varphi $ be a $Y$-periodic $C^\infty$ function in $\R^d$
such that $\varphi =0$ in $\R^d \setminus \omega_{1, \eta}$.
We need to show that
$$
\int_Y \nabla \psi_\eta \cdot \nabla \varphi\, dx
=\int_Y F_\eta \varphi\, dx
-\int_Y f_\eta \cdot \nabla \varphi\, dx
$$
for some $F_\eta$ and $\eta_\eta$ satisfying \eqref{p-2}.
To this end, 
observe that
$$
\aligned
\int_Y \nabla \psi_\eta \cdot \nabla \varphi\, dx
 & =\int_{B(0, 1/3)\setminus B(0, 1/4)} \nabla \psi_\eta \cdot \nabla \varphi\, dx
 +\int_{B(0, 1/4) \setminus \eta T} \nabla \psi_\eta \cdot \nabla \varphi\, dx\\
 &=I_1 +I_2.
 \endaligned
 $$
 For $I_1$, recall that 
  \begin{equation}\label{e-p-4}
   |\nabla \psi_\eta (x)|\le C \eta^{d-2} \qquad \text{ for } 
 x\in B(0, 1/3) \setminus B(0, 1/4).
 \end{equation}
 
 To handle $I_2$,  using  $\varphi=0$ on $\partial (\eta T)$, we may write
 \begin{equation}\label{e-p-5}
 \aligned
 I_2 & = \int_{\partial B(0, 1/4)} \frac{\partial \psi_\eta}{\partial n} \varphi\, d\sigma,\\
 \endaligned
 \end{equation}
where we also used the fact that $\psi_\eta$ is harmonic in $B(0, 1/4)\setminus {\eta T}$.
Let 
$$
g=\frac{\partial \psi_\eta}{\partial n}= \eta^{-1} n\cdot \nabla \phi_* (x/\eta)
$$
 on $\partial B(0, 1/4)$.
By \eqref{e-p-2},  $ |g | + |\nabla g|\le C \eta^{d-2}$.
Hence, there exists $G \in C^1(B(0, 1/4))$ such that
$G=g$ on $\partial B(0, 1/4)$ and
$|G | + |\nabla G |\le C \eta^{d-2}$ in $B(0, 1/4)$.
It follows that 
$$
\int_{\partial B(0, r )} g \varphi\, d\sigma
=\frac{1}{r} \int_{B(0, r)}  \{ d G + x\cdot \nabla G \} \varphi \, dx
+ \frac{1}{r} \int_{B(0, r)} G (x \cdot \nabla \varphi)\,dx,
$$
where $r=(1/4)$. 
This, together with \eqref{e-p-4}, yields \eqref{p-1} and \eqref{p-2}.

The proof for the case $d=2$ is similar. Indeed,  note that
$$
\aligned
\int_Y \nabla \psi_\eta \cdot \nabla \varphi\, dx
& =\int_{B(0, 1/3)\setminus B(0, d\eta)} \nabla \psi_\eta \cdot \nabla \varphi\, dx\\
& =\int_{\partial B(0, 1/3)} \frac{\partial \psi_\eta}{\partial n} \varphi\, d\sigma\\
&=\frac{3}{\ln (1/3) -\ln (d\eta)}
\int_{\partial B(0, 1/3)} \varphi\, d\sigma\\
&=\frac{9}{\ln (1/3) -\ln (d\eta)}
\int_{B(0, 1/3)}  \left( 2\varphi + x\cdot \nabla \varphi \right) dx,
\endaligned
$$
which yields the estimate \eqref{p-3}.
\end{proof}

We are now in a position to give the proof of Theorem \ref{thm-local}

\begin{proof}[\bf Proof of Theorem \ref{thm-local} ]
Let $u$ be a solution of \eqref{l-0}.
Let $\psi_\eta$ be defined by \eqref{psi}-\eqref{2d-p}.
Note that for any $\alpha \in \R$,
we have $u-\alpha \psi_\eta =0$ on $\partial (\eta T)$ and
$$
-\Delta (u-\alpha \psi_\eta)
=(F-\alpha F_\eta) +\text{\rm div}(f-\alpha f_\eta)
$$
in $\tY\setminus \eta T$.
It follows by Lemma \ref{lemma-l-1} that
\begin{equation}\label{local-a}
\aligned
\left(\int_{Y \setminus \eta T}
|\nabla u |^p\, dx \right)^{1/p}
& \le  |\alpha| \left(\int_{Y \setminus \eta T} |\nabla \psi_\eta|^p\ dx \right)^{1/[p} + C \Phi_p(\eta^{-1}) 
|\alpha | \left( \| F_\eta \|_\infty + \| f_\eta\|_\infty \right)\\
&\qquad\qquad
+ C \Phi_p(\eta^{-1}) 
\left(\int_{\tY \setminus \eta T} \left( |F|^p + |f|^p \right)\, dx \right)^{1/p}\\
& \qquad \qquad
+ C \Phi_p (\eta^{-1}) 
\left(\int_{\tY \setminus B(0, 1/3) } |u-\alpha|^2 \, dx \right)^{1/2},
\endaligned
\end{equation}
where we have used the fact $\psi_\eta =1$ in $\tY \setminus B(0, 1/3) $.
By Lemmas \ref{p0-lemma} and \ref{p-lemma}, if $d\ge 3$,
the first two terms in the right-hand side of \eqref{local-a} are bounded by
$$
C |\alpha|  \eta^{\frac{d}{p}-1}
+ C |\alpha| \Phi_p (\eta^{-1}) \eta^{d-2}
\le C |\alpha | \eta^{\frac{d}{p}-1}.
$$
This, together with \eqref{local-a}, gives \eqref{l-1}.
Similarly, if $d=2$, the first two terms in the right-hand side of \eqref{local-a} are bounded by
$$
C |\alpha | \eta^{\frac{2}{p}-1} |\ln \eta|^{-1},
$$
which yields \eqref{l-1-2d}.
\end{proof}



\section{Proofs of Theorems \ref{main-thm-1} and \ref{main-thm-2}} \label{section-p}

We begin with an estimate for $\| u\|_{L^p(\omega_{\e, \eta})}$.

\begin{lemma}
Let  $1< p< \infty$.
For any $F\in L^p (\omega_{\e, \eta})$ and $f\in L^p(\omega_{\e, \eta}; \R^d)$, the Dirichlet problem 
\eqref{D-01} has a unique solution in $W_0^{1, p}(\Omega_{\e, \eta})$. Moreover, if $2\le p< \infty$, the solution 
satisfies 
\begin{equation}\label{m-1}
\| u\|_{L^p(\omega_{\e, \eta})}
\le C\left\{ \e^2 \eta^{2-d} \| F \|_{L^p(\omega_{\e, \eta})}
+  \e \eta^{1-\frac{d}{2}} \| f\|_{L^p(\omega_{\e, \eta})} \right\} 
\end{equation}
for $d\ge 3$, and
\begin{equation}\label{m-2}
\| u\|_{L^p(\omega_{\e, \eta})}
\le C \left\{ \e^2 |\ln (\eta/2)| \| F \|_{L^p(\omega_{\e, \eta})}
+ \e |\ln (\eta/2)|^{1/2} \| f \|_{L^p(\omega_{\e, \eta})} \right\}
\end{equation}
for $d=2$. The constant $C$ depends only on $d$, $p$ and $c_0$.
\end{lemma}

\begin{proof}
The existence and uniqueness of the solution are known \cite{Masmoudi-2004, Shen-2022}.
The estimates \eqref{m-1}-\eqref{m-2}  for $2\le p<\infty$ were proved in \cite[Theorem 3.3]{Shen-2022} in a general non-periodic setting.
In particular,  the $C^1$ assumption on $T$ is not needed.
\end{proof}

Next, we consider the case $F=0$.

\begin{thm}\label{thm-m-1}
Let $1< p< \infty$. 
For any $f\in L^p(\omega_{\e, \eta}; \R^d )$, the solution of  the Dirichlet problem,
\begin{equation}\label{D-m-1}
-\Delta u=\text{\rm div}(f) \quad \text{ in } \omega_{\e, \eta} \quad
\text{  and } \quad 
u=0 \quad  \text{ on } \partial\omega_{\e, \eta},
\end{equation}
in $W_0^{1, p}(\omega_{\e, \eta})$  satisfies the estimate,
\begin{equation}\label{m-3}
\|\nabla u\|_{L^p(\omega_{\e, \eta})}
\le C \eta^{-d |\frac12 -\frac{1}{p}|} \| f \|_{L^p(\omega_{\e, \eta})},
\end{equation}
for $d\ge 3$, and
\begin{equation}\label{m-3a}
\|\nabla u\|_{L^p(\omega_{\e, \eta})}
\le C \eta^{-2 |\frac12 -\frac{1}{p}|} |\ln (\eta/2)|^{-\frac12} \| f \|_{L^p(\omega_{\e, \eta})},
\end{equation}
for $d=2$ and $p\neq 2$,
where $C$ depends only on $d$, $p$ and $T$.
\end{thm}

\begin{proof}
By rescaling and duality  we may assume that $\e=1$ and $p>2$.
Moreover, we only need to prove the estimates \eqref{m-3}-\eqref{m-3a}  for $\eta>0$ sufficiently small.

We first consider the case $d\ge 3$.
Let $u$ be a solution of \eqref{D-m-1} with $\e=1$.
It follows by Theorem \ref{thm-local} that
$$
\aligned
\int_{k+ (Y\setminus \eta T)}
|\nabla u|^p\, dx
 & \le C |\alpha |^p \eta^{d-p}
+ C  [\Phi_p (\eta^{-1})]^p
\int_{k+ (\tY\setminus \eta T)} |f|^p\, dx\\
& \qquad \qquad
+ C[ \Phi_p (\eta^{-1})]^p
\left(
\int_{k+ (\tY \setminus B(0, 1/3) )} | u-\alpha |^2 \, dx \right)^{p/2}
\endaligned
$$
for any $k\in \mathbb{Z}^d$ and $\alpha \in \R$.
Choose
$$
\alpha = \fint_{k+ (\tY\setminus B(0, 1/3) )} u\, dx.
$$
By using the Poincar\'e inequality we obtain 
$$
\aligned
\int_{k+ (Y\setminus \eta T)}
|\nabla u|^p\, dx
 & \le C  \eta^{d-p} \int_{k + (\tY \setminus \eta T)} |u|^p\, dx 
+ C [ \Phi_p (\eta^{-1})]^p \int_{k+ (\tY\setminus \eta T)} |f|^p\, dx\\
& \qquad \qquad
+ C [\Phi_p (\eta^{-1})]^p
\left(\int_{k+(\tY\setminus \eta T)} |\nabla u|^2\right)^{p/2} \\
& \le C  \eta^{d-p} \int_{k + (\tY \setminus \eta T)} |u|^p\, dx 
+ C [ \Phi_p (\eta^{-1})]^p \int_{k+ (\tY\setminus \eta T)} |f|^p\, dx\\
& \qquad \qquad
+ C[ \Phi_p (\eta^{-1})]^p
\int_{k+Y} |S_{1, \eta} (0, f) |^p,
\endaligned
$$
where the operator $S_{1, \eta}$ is defined by \eqref{S}.
By summing over $k \in \mathbb{Z}^d$ we deduce that
$$
\aligned
\| \nabla u \|_{L^p(\omega_{1, \eta})}
 & \le C \eta^{\frac{d}{p}-1} \| u \|_{L^p(\omega_{1, \eta})}
+ C \Phi_p (\eta^{-1})  \left\{
\| f\|_{L^p(\omega_{1, \eta})}
+ \| S_{1, \eta} (0, f) \|_{L^p(\R^d)} \right\}\\
& \le C  \eta^{-d (\frac12-\frac{1}{p})} 
\| f\|_{L^p(\omega_{1, \eta})},
\endaligned
$$
where we have used \eqref{m-1} and \eqref{S-01} as well as the observation $\Phi_p (\eta^{-1}) \le C \eta^{\frac{d}{p}-\frac{d}{2}}$
in the case $d\ge 3$
 for the last inequality.
This gives \eqref{m-3} with $\e=1$ and $p>2$ for the case $d\ge 3$.

The proof for the case $d=2$ is similar. Using \eqref{l-1-2d}, we obtain
$$
\| \nabla u \|_{L^p(\omega_{1, \eta})}
\le C \eta^{\frac{2}{p}-1} 
 |\ln \eta|^{-1}\left\{  \| u \|_{L^p(\omega_{1, \eta})}
+
\| f\|_{L^p(\omega_{1, \eta})} + \| S_{1, \eta} (0, f) \|_{L^p(\R^d) } \right\}.
$$
The desired estimate then follows from \eqref{m-2} and \eqref{S-01}.
\end{proof}

We now consider the case $f=0$.

\begin{thm}\label{thm-m-2}
Let $1<p< \infty$.
For any $F\in L^p(\omega_{\e, \eta})$, the solution of the Dirichlet problem,
\begin{equation}\label{D-m-2}
-\Delta u=F \quad \text{ in } \omega_{\e, \eta} \quad \text{ and } \quad
u=0 \quad \text{ on } \partial\omega_{\e, \eta},
\end{equation}
 in $W_0^{1, p}(\omega_{\e, \eta} )$  satisfies the estimate,
\begin{equation}\label{m-20}
\|\nabla u \|_{L^p(\omega_{\e, \eta})}
\le \left\{
\aligned
& C \e \eta^{1-\frac{d}{2}} \| F \|_{L^p(\omega_{\e, \eta})} & \quad & \text{ for } 1< p\le 2,\\
& C \e\eta^{1-d +\frac{d}{p}} \| F \|_{L^p(\omega_{\e, \eta})} & \quad & \text{ for } 2< p< \infty
\endaligned
\right.
\end{equation}
for $d\ge 3$, and
\begin{equation}\label{m-20a}
\|\nabla u \|_{L^p(\omega_{\e, \eta})}
\le \left\{
\aligned
& C \e  |\ln (\eta/2)|^{\frac12} \| F \|_{L^p(\omega_{\e, \eta})} & \quad & \text{ for } 1< p\le 2,\\
& C \e\eta^{-1 +\frac{2}{p}}  \| F \|_{L^p(\omega_{\e, \eta})} & \quad & \text{ for } 2< p< \infty
\endaligned
\right.
\end{equation}
for $d=2$.
\end{thm}

\begin{proof}
The case $1<p\le 2$ for $d\ge 2$  was proved in \cite[Theorem 5.2]{Shen-2022} in a general non-periodic setting.
To treat the case $2<p< \infty$, we may assume $\e=1$ by rescaling.
Suppose $d\ge 3$ and $u$ is a solution of \eqref{D-m-2}.
As in the proof of Theorem \ref{thm-m-1}, using
 Theorem \ref{thm-local}, we may deduce by summation  that 
$$
\aligned
\|\nabla u \|_{L^p(\omega_{1, \eta})}
 & \le C \eta^{\frac{d}{p}-1} \| u\|_{L^p (\omega_{1, \eta})}
+ C \Phi_p (\eta^{-1}) 
\left\{ \| F \|_{L^p(\omega_{1, \eta})}
+ \| S_{1, \eta} (F, 0) \|_{L^p(\omega_{1, \eta})} \right\}\\
&\le C \eta^{1-d + \frac{d}{p}} \| F \|_{L^p(\omega_{1, \eta})}
+  C \Phi_p (\eta^{-1}) 
\left\{ \| F \|_{L^p(\omega_{1, \eta})}
+ \eta^{1-\frac{d}{2}} \| F \|_{L^p(\omega_{1, \eta})} \right\}\\
&\le C \eta^{1-d +\frac{d}{p}} \| F \|_{L^p(\omega_{1, \eta})},
\endaligned
$$
where we have used \eqref{m-1} and \eqref{S-21} for the second inequality.
The proof for the case $d=2$ is similar.
By Theorem \ref{thm-local}, we obtain 
$$
\|\nabla u\|_{L^p(\omega_{1, \eta})}
\le C \eta^{\frac{2}{p}-1} |\ln \eta|^{-1} \| u \|_{L^p(\omega_{1, \eta})}
+ C \eta^{ \frac{2}{p}-1} |\ln \eta|^{-1} 
\left\{ \| F \|_{L^p(\omega_{1, \eta})}
+ \| S_{1, \eta} (F, 0)\|_{L^p(\omega_{1, \eta})} \right\}
$$ 
which, together with  \eqref{m-2} and \eqref{S-21}, yields \eqref{m-20a} for $2<p<\infty$.
\end{proof}

For $1< p< \infty$ and $\e, \eta \in (0, 1]$,
let $A_p(\e, \eta)$, $B_p (\e, \eta)$, $C_p(\e, \eta)$ and $D_p(\e, \eta)$ be the smallest constants for which 
the inequalities  \eqref{A-B} and \eqref{C-D} hold.
Clearly, $A_2 (\e, \eta)\le 1$.
By duality, $C_p(\e, \eta) =B_{p^\prime} (\e, \eta)$, where $p^\prime =\frac{p}{p-1}$ 
(see \cite{Shen-2022}).
It follows from Theorems \ref{thm-m-1} and \ref{thm-m-2}  that
\begin{equation}\label{A-p-f}
A_p (\e, \eta)
\le \left\{
\aligned
 &C \eta^{-d |\frac12 -\frac{1}{p}|} & \quad & \text{ if } d\ge 3,\\
 & C \eta^{-2 |\frac12-\frac{1}{p}|} |\ln (\eta/2)|^{-\frac12} &\quad &  \text{ if } d=2 \text{ and } p\neq 2,
 \endaligned
 \right.
 \end{equation}
 and
 \begin{equation}\label{B-p-f}
 B_p(\e, \eta)
 =C_{p^\prime}(\e, \eta)
\le
\left\{
\aligned
& C \e \eta^{1-\frac{d}{2}} & \quad & \text{ if } d\ge 3 \text{ and } 1< p\le 2,\\
&C \e |\ln (\eta/2)|^{\frac12} & \quad & \text{ if } d=2 \text{ and } 1< p\le 2, \\
& C \e \eta^{-1 +\frac{2}{p}} & \quad & \text{ if }  d\ge 2 \text{ and } 2< p< \infty,
\endaligned
\right.
\end{equation}
 where $C$ depends only on $d$, $p$ and $T$.
 Furthermore, it was proved in \cite{Shen-2022} that 
 \begin{equation}\label{D-p-f}
 D_p (\e, \eta)
 \le  \left\{
 \aligned
&  C \e^2 \eta^{2-d} & \quad & \text{ if } d\ge 3,\\
& C \e^2 |\ln (\eta/2)| & \quad & \text{ if } d=2.
\endaligned
\right.
\end{equation}

\begin{proof}[\bf Proofs of Theorems \ref{main-thm-1} and \ref{main-thm-2}]
The estimates \eqref{m-e-1} and \eqref{m-e-1a} follow from \eqref{A-p-f} and \eqref{B-p-f} by linearity, 
while \eqref{m-e-2} and \eqref{m-e-2a} follow from \eqref{B-p-f} and \eqref{D-p-f}.
As we mentioned in the introduction, the sharpness of the estimates \eqref{A-p-f}-\eqref{D-p-f}   was proved in \cite{Shen-2022}.
\end{proof}

\noindent{\bf Declaration of Interest: none.}

 \bibliographystyle{amsplain}
 
\bibliography{Shen-Wallace.bbl}

\bigskip

\begin{flushleft}

Zhongwei Shen (corresponding author),
Department of Mathematics,
University of Kentucky,
Lexington, Kentucky 40506,
USA.
E-mail: zshen2@uky.edu
\end{flushleft}

\begin{flushleft}

Jamison Wallace, 
Department of Mathematics,
University of Kentucky,
Lexington, Kentucky 40506,
USA.
E-mail: Jamison.Wallace@uky.edu
\end{flushleft}

\bigskip

\end{document}